\documentclass[11pt]{article}
\usepackage{amssymb}
\usepackage{amsthm}
\usepackage{amsmath}
\usepackage{latexsym}
\usepackage{graphics}
\usepackage{xspace}
\usepackage{psfrag}
\usepackage{epsfig}
\usepackage{pst-all}
\usepackage{algorithmic}
\usepackage{enumerate}

\newtheorem{theorem}{Theorem}[section]
\newtheorem{lemma}[theorem]{Lemma}

\newtheorem{prop}[theorem]{Proposition}
\newtheorem{proposition}[theorem]{Proposition}
\newtheorem{coro}[theorem]{Corollary}

\newtheorem{definition}[theorem]{Definition}

\newcommand{\ER}{Erd{\"o}s-R\'{e}nyi }
\newcommand{\avg}{\text{\rm avg} }
\newcommand{\pr}{\mathbb{P}}
\newcommand{\E}{\mathbb{E}}
\newcommand{\R}{\mathbb{R}}

\newcommand{\G}{\mathbb{G}}

\newcommand{\Anv}{\ensuremath{\operatorname{ANOVA}}\xspace}

\newcommand{\LAS}{\mathcal{LAS}}
\newcommand{\IGP}{\mathcal{IGP}}
\newcommand{\Amatrix}{\mathbf{A}}
\newcommand{\Cmatrix}{\mathbf{C}}
\newcommand{\U}{\mathbf{U}}
\newcommand{\Zmatrix}{\mathbf{Z}}
\newcommand{\Dmatrix}{\mathbf{D}}

\newcommand{\Ave}{\ensuremath{\operatorname{Ave}}\xspace}

\newcommand{\ignore}[1]{\relax}

\begin{document}

\title{\Large Finding a Large Submatrix of a Gaussian Random Matrix
}

\author{
{\sf David Gamarnik}\thanks{MIT; e-mail: {\tt gamarnik@mit.edu}.Research supported  by the NSF grants CMMI-1335155.}
\and
{\sf Quan Li}\thanks{MIT; e-mail: {\tt quanli@mit.edu}}
}

\date{\vspace{-5ex}}

\maketitle

\begin{abstract}
We consider the problem of finding a $k\times k$ submatrix of an $n\times n$ matrix with i.i.d. standard Gaussian entries, which has a large average entry.
It was shown in \cite{bhamidi2012energy} using non-constructive methods that the largest average value of a $k\times k$ submatrix
is $2(1+o(1))\sqrt{\log n/k}$ with high probability (w.h.p.)
when $k = O(\log n / \log \log n)$. In the same
paper an evidence was provided that a natural greedy algorithm called Largest Average Submatrix ($\LAS$)
for a constant $k$ should produce a matrix with average entry at most $(1+o(1))\sqrt{2\log n/k}$, namely
approximately $\sqrt{2}$ smaller,
though
no formal proof of this fact was provided.

In this paper we show that the matrix produced by the $\LAS$ algorithm is indeed $(1+o(1))\sqrt{2\log n/k}$ w.h.p. when $k$ is constant and $n$ grows. Then by
drawing an analogy with the problem of finding cliques in random graphs, we  propose a simple greedy algorithm which produces a $k\times k$ matrix with asymptotically
the same average value $(1+o(1))\sqrt{2\log n/k}$ w.h.p., for $k=o(\log n)$.  Since the greedy algorithm is the best known algorithm for finding cliques in random graphs, it
is tempting to believe that beating the factor $\sqrt{2}$ performance gap suffered by both algorithms might be very challenging. Surprisingly, we show the existence of  a very simple algorithm
which produces a $k\times k$ matrix with average value $(1+o_k(1))(4/3)\sqrt{2\log n/k}$ for in fact $k=o(n)$.

To get an insight into the algorithmic hardness of this problem, and motivated by methods originating in the theory of spin glasses, we conduct the so-called expected overlap
analysis of matrices with average value asymptotically  $(1+o(1))\alpha\sqrt{2\log n/k}$ for a fixed value $\alpha\in [1,\sqrt{2}]$.
The overlap corresponds to the number of common rows and common columns for
pairs of matrices achieving this value (see the paper for details). We discover numerically an intriguing phase
transition at $\alpha^*\triangleq 5\sqrt{2}/(3\sqrt{3})\approx 1.3608..\in [4/3,\sqrt{2}]$:
when $\alpha<\alpha^*$ the space of overlaps is a continuous subset of $[0,1]^2$, whereas $\alpha=\alpha^*$ marks the onset of discontinuity, and as a result
the model exhibits the \emph{Overlap Gap Property (OGP)} when $\alpha>\alpha^*$, appropriately defined.
We conjecture that \emph{OGP} observed for $\alpha>\alpha^*$ also marks the onset of the algorithmic hardness - no polynomial time algorithm exists for finding
matrices with average value at least $(1+o(1))\alpha \sqrt{2 \log n/k}$, when $\alpha>\alpha^*$ and $k$ is a growing function of $n$.
\end{abstract}

\section{Introduction}
We consider the algorithmic problem of finding a submatrix of a given random matrix such that the average value of the submatrix is appropriately large.
Specifically, consider an $n\times n$ matrix $\Cmatrix^n$
with i.i.d. standard Gaussian entries. Given $k\le n$, the goal is to find algorithmically a $k\times k$ submatrix $\Amatrix$ of $\Cmatrix^n$
(not necessarily principal) with average entry as large as possible. The problem has motivations in several areas, including biomedicine, genomics and social
networks~\cite{shabalin2009finding},\cite{madeira2004biclustering},\cite{fortunato2010community}. The search of such matrices is called ''bi-clustering''~\cite{madeira2004biclustering}.
The problem of finding asymptotically the largest average entry of $k\times k$ submatrices of $\Cmatrix^n$ was recently studied
by Bhamidi~et.al.~\cite{bhamidi2012energy} (see also~\cite{sun2013maximal} for a related study) and questions arising in this paper constitute the motivation for our work.
It was shown in~\cite{bhamidi2012energy} using non-constructive methods
that the largest achievable average entry of a $k\times k$ submatrix of $\Cmatrix^n$ is asymptotically with high probability (w.h.p.) $(1+o(1))2\sqrt{\log n/k}$ when
$n$ grows and
$k=O(\log n/\log\log n)$ (a more refined distributional result is obtained). Here $o(1)$ denotes a function converging to zero as $n\rightarrow\infty$.
Furthermore, the authors consider the asymptotic value and the number of so-called locally maximum
matrices.
A $k\times k$ matrix $\Amatrix$ is locally maximal if every $k\times k$ matrix of $\Cmatrix^n$ with the same set of rows as $\Amatrix$ has a smaller average value than
that of $\Amatrix$
and every $k\times k$ matrix of $\Cmatrix^n$ with the same set of columns as $\Amatrix$ has a smaller average value than that of $\Amatrix$. Such local maxima are natural
objects arising as terminal matrices produced by a simple iterative procedure called Large Average Submatrix ($\LAS$), designed for finding a matrix with a large average
entry. $\LAS$ proceeds by starting with an arbitrary $k\times k$ submatrix $\Amatrix_0$ and finding a  matrix $\Amatrix_1$ sharing the same set of rows with
$\Amatrix_0$ which has the largest average value. The procedure is then repeated for $\Amatrix_1$ by searching through columns of $\Amatrix_1$ and identifying
the best matrix $\Amatrix_2$. The iterations proceed while possible and at the end some locally maximum matrix $\Amatrix_{\LAS}$ is produced as the output. The authors show that
when $k$ is constant, the
majority of locally maximum matrices of $\Cmatrix^n$ have the asymptotic value $(1+o(1))\sqrt{2\log n/k}$ w.h.p. as $n$ grows,
thus factor $\sqrt{2}$ smaller than the global optimum. Motivated by this finding,
the authors suggest that the outcome of the $\LAS$ algorithm should be also factor $\sqrt{2}$ smaller than the global optimum, however one cannot deduce this from
the result of~\cite{bhamidi2012energy} since it is not ruled out that $\LAS$ is clever enough to find a ``rare'' locally maximum matrix with a significantly larger average value
than $\sqrt{2\log n/k}$.

The main result of this paper is the confirmation of this conjecture for the case of constant $k$:  the $\LAS$ algorithm produces a matrix with asymptotic average
value $(1+o(1))\sqrt{2\log n/k}$ w.h.p. We further establish that the number of iterations of the $\LAS$ algorithm is stochastically bounded as $n$ grows.
The proof of this result
is fairly involved and proceeds by a careful conditioning argument. In particular, we show that for fixed $r$, conditioned on the event
that $\LAS$ succeeded in iterating at least $r$ steps, the probability distribution of the ``new best matrix'' which will be used in constructing the matrix for the next iteration
is very close to the largest matrix in the $k\times n$ strip of $\Cmatrix^n$, and which is known to have asymptotic average value of $\sqrt{2\log n/k}$ due to result
in~\cite{bhamidi2012energy}. Then we show that the matrix produced in step $r$ and the best matrix in the $k\times n$ strip among the unseen entries
are asymptotically independent.
Using this we show that given that $\LAS$ proceeded with $r$ steps the likelihood it proceeds with the next $r+2k+4$ steps
is at most some value $\psi<1$ which is bounded away from $1$ as $n$ grows. As a result the number of steps of $\LAS$ is upper bounded by a geometrically decaying function
and thus is stochastically bounded. We use this as a key result in computing the average value produced by $\LAS$, again relying on the asymptotic independence
and the average value of the $k\times n$ strip dominant submatrix.

As it was observed already in~\cite{bhamidi2012energy}, the factor $\sqrt{2}$ gap between the global optimum and the performance of $\LAS$ is reminiscent of a similar
gap arising in studying of  largest cliques of  random graphs. Arguably, one of the oldest algorithmic open problems in the field of random graph is the problem
of finding a largest clique (a fully connected subgraph) of a random \ER graph $\G(n,p)$, when $p$ is at least $n^{-1+\delta}$ for some positive constant $\delta$.
It is known that the value is asymptotically $2\log n/(-\log p)$ and a simple greedy procedure produces a clique with size $\log n/(-\log p)$, namely
factor $2$ smaller than the global optimum. A similar
result holds for the bi-partite \ER graph: the largest clique is asymptotically $2\log n/(-\log p)$ and the greedy algorithm produces
a (bi-partite) clique of size asymptotically $\log n/(-\log p)$. Karp in his 1976 paper~\cite{karp1976probabilistic}
challenged to find a better algorithm leading to a clique with size say $(1+\epsilon)\log_2 n$ and this problem remains open. The factor $\sqrt{2}$ appearing
in our context is then arguably an analogue of the factor $2$ arising in the context of the clique problem in $\G(n,p)$. In order to further investigate the
possible connection between the two problems, we propose the following
simple algorithm for finding a submatrix of $\Cmatrix^n$ with a large average entry. Fix a positive threshold $\theta$ and consider the random $0,1$ matrix
$\Cmatrix^n_\theta$ obtained by thresholding each Gaussian entry of $\Cmatrix^n$ at $\theta$. Clearly $\Cmatrix^n_\theta$ is an adjacency matrix
of a bi-partite \ER graph $\G(n,p_\theta)$, where $p_\theta=\pr(Z>\theta)$ and $Z$ is a standard Gaussian random variable. Observe that any $k\times k$ clique
of $\G(n,p_\theta)$ corresponds to a $k\times k$ submatrix of $\Cmatrix^n$ with \emph{each} entry at least $\theta$. Thus any polynomial time algorithm
for finding a clique in $\G(n,p_\theta)$ which results in a $k\times k$ clique w.h.p. immediately gives a matrix with average value at least $\theta$ w.h.p. Consider
the greedy algorithm and adjust $\theta$ so that the size of the clique is at least $k$ on each side. Reverse engineering $\theta$ from such $k$,
one can find that $\theta\approx \sqrt{2\log n/k}$ with $p\approx \exp(-\theta^2/2)=n^{1\over k}$ (see the next section for a simple derivation of this fact).
Namely, both $\LAS$ and the greedy algorithm have the same asymptotic power!
(Note, however, that this analysis extends beyond the $k=O(1)$ unlike our analysis of the $\LAS$ algorithm).

In light of these connections with studying cliques in random graphs and the apparent
failure to bridge the factor $2$ gaps for cliques, one might suspect that $\sqrt{2}$ is equally challenging to beat for the maximum submatrix problem.
Perhaps surprisingly, we establish that this is not the case and
construct a very simple
algorithm, both in terms of analysis and implementation, which construct a submatrix with average value asymptotically $(1+o_k(1)(4/3)\sqrt{2\log n/k}$
for $k=o(\log^2 n/(\log\log n)^2)$.
Here $o_k(1)$ denotes a function decaying to zero as $k$ increases.
The algorithm
proceeds by starting with one entry and iteratively building a sequence of $r\times r$ and $r\times (r+1)$ matrices for $r=1,\ldots,k$ in a simple greedy fashion.
We call this algorithm Incremental Greedy Procedure ($\IGP$), referring to the incremental increase of the matrix size. No immediate simple modifications of $\IGP$
led to the improvement of the $4/3$ factor, unfortunately.

The discussion above raises the following question:
where is the true algorithmic hardness threshold value for the maximum submatrix problem if such exists? Short of proving some formal hardness of this
problem, which seems out of reach for the currently known techniques both for this problem and the clique problem for $\G(n,p)$, we propose
an approach which indirectly suggests the hardness regime for this problem, and this is our last contribution.
Specifically, our last contribution is the conjecture for this value based on the
\emph{Overlap Gap Property} (OGP) which originates in the theory of spin glasses and which we adopt here in the context of our problem in the following way.
We fix $\alpha\in (1,\sqrt{2})$ and let $\mathcal{L}(\alpha)$ denote the set of matrices with average value asymptotically $\alpha\sqrt{2\log n/k}$. Thus $\alpha$
conveniently parametrizes the range between the achievable value on the one hand, namely $\alpha=1$ for $\LAS$ and greedy algorithms, $\alpha=4/3$ for the $\IGP$, and
$\alpha=\sqrt{2}$ for the global optimum on the other hand. For every pair of matrices $\Amatrix_1,\Amatrix_2\in\mathcal{L}(\alpha)$ with row sets $I_1,I_2$ and column
sets $J_1,J_2$ respectively, let $x(\Amatrix_1,\Amatrix_2)=|I_1\cap I_2|/k, y(\Amatrix_1,\Amatrix_2)=|J_1\cap J_2|/k$. Namely $x$ and $y$
are the normalized counts of the common rows and common columns for the two matrices. For every $(x,y)\in [0,1]^2$ we consider the expected
number  of pairs $\Amatrix_1,\Amatrix_2$ such that $x(\Amatrix_1,\Amatrix_2)\approx x,y(\Amatrix_1,\Amatrix_2)\approx y$, in some appropriate sense to be made precise.
We compute this expectation asymptotically. We define $R(x,y)=0$ if such an expectation converges to zero as $n\rightarrow\infty$ and $=1$ otherwise.
Thus the set $\mathcal{R}(\alpha)\triangleq \{(x,y): R(x,y)=1\}$ describes the set of achievable in expecation overlaps of pairs of matrices with average value
$\alpha\sqrt{2\log n/k}$. At $\alpha^*\triangleq 5\sqrt{2}/(3\sqrt{3})\approx 1.3608..$ we observe an interesting phase transition -- the set
$\mathcal{R}(\alpha)$ is connected for $\alpha<\alpha^*$, and is disconnected for $\alpha>\alpha^*$ (see Figures~\ref{fig:figure3}). Namely, for $\alpha>\alpha^*$
the model exhibits the OGP. Namely, the overlaps of two matrices belong to one of the two disconnected regions.

Motivated by this observation,
we conjecture that the problem  of finding a matrix with the corresponding value $\alpha>\alpha^*$ is not-polynomially solvable when $k$ grows. In fact, by considering
multi-overlaps instead of pairwise overlaps, (which we intend to research in future), we conjecture that this hardness threshold might be even lower than $\alpha^*$.
The link between OGP and algorithmic hardness has been suggested and partially established in the context of sparse random constraint satisfaction problems,
such as random K-SAT problem, coloring of sparse \ER problem and the problem of finding a largest independent set
of a sparse \ER graph problem~\cite{AchlioptasCojaOghlanRicciTersenghi},\cite{achlioptas2008algorithmic},\cite{coja2011independent},
\cite{gamarnik2014limits},\cite{rahman2014local},\cite{gamarnik2014performance},\cite{montanari2015finding}.
Many of these problems exhibit an apparent gap between the best existential values and the best values found by known algorithms, very similar
in spirit to the gaps $2, \sqrt{2}$ etc. discussed above in our context.
For example, the largest independent set of a random $d$-regular graph
normalized by the number of nodes  is known to be asymptotically $2\log d/d$ as $d$ increases, while the best algorithm can produce sets of size only
$\log d/d$ again as $d$ increases. As shown in~\cite{coja2011independent},\cite{gamarnik2014limits} and \cite{rahman2014local} the threshold $\log d/d$ marks
the onset of a certain version of OGP. Furthermore, \cite{coja2011independent},\cite{gamarnik2014limits} show that OGP
is the bottleneck for a certain class of algorithms, namely local algorithms (appropriately defined). A key step observed in~\cite{rahman2014local}
is that the threshold for multioverlap version of the OGP, namely considering m-tuples of solutions as opposed to pairs of solutions as we do in this paper, lowers
the phase transition point. The multioverlap version of OGP  was also a key step in~\cite{gamarnik2014performance} in the context of random Not-All-Equal-K-SAT (NAE-K-SAT)
problem which also exhibits a marked gap between the regime where the existence of a feasible solution is known and the regime where such a solution can be found by known algorithms.
The OGP for largest submatrix problem thus adds to the growing class of optimization problems with random input which exhibit a significant gap between the globally optimal
solution and what is achievable by currently known algorithmic methods.

The remainder of the paper is structured as follows. In the next section we formally state our four main results: the one regarding the performance of $\LAS$,
the one regarding the performance of the greedy algorithm by reduction to random bi-partite graphs, the result regarding the performance of $\IGP$, and finally
the result regarding the OGP. The same section provides a short proof for the result regarding the greedy algorithm.
Section~\ref{section:IGP} is devoted to the proof of the result regarding the performance of $\IGP$. Section~\ref{section:OGP} is devoted to the
proof of the result discussing OGP, and Section~\ref{section:LAS} (which is the most technically involved part of the paper) is devoted to the proof of the
result regarding the performance of the $\LAS$ algorithm. We conclude in Section~\ref{section:Conclusions} with some open questions.

We close this section with some notational convention. We use standard notations $o(\cdot)$, $O(\cdot)$ and $\Theta(\cdot)$ with
respect to $n \rightarrow \infty$. $o_k(1)$ denotes a function $f(k)$ satisfying $\lim_{k \rightarrow \infty} f(k) = 0$.
Given a positive integer $n$, $[n]$ stands for the set of integers $1,\ldots,n$.
Given a matrix $A$, $A^T$ denotes its transpose. $\Rightarrow$ denotes weak convergence. $\,{\buildrel d \over =}\,$ denotes equality in distribution. A complement of event $\mathcal{A}$ is denoted by $\mathcal{A}^c$. For two events
$\mathcal{A}$ and $\mathcal{B}$ we write $\mathcal{A}\cap \mathcal{B}$ and $\mathcal{A}\cup \mathcal{B}$ for the intersection (conjunction)
and the union (disjunction) of the two events, respectively. When conditioning on the event $\mathcal{A}\cap \mathcal{B}$ we will often write
$\pr\left(\cdot | \mathcal{A},\mathcal{B}\right)$ in place of $\pr\left(\cdot | \mathcal{A}\cap \mathcal{B}\right)$.

\section{Main Results}\label{our_result}
In this section we formally describe the algorithms we analyze in this paper and state our main results.
Given an $n\times n$ matrix $A$ and subsets $I\subset [n], J\subset [n]$ we denote by $A_{I,J}$ the submatrix of $A$ indexed by rows $I$ and columns $J$.
When $I$ consist of a single row $i$, we use $A_{i,J}$ in place of a more proper $A_{\{i\},J}$.
Given any $m_1\times m_2$ matrix $B$, let $\Ave(B)\triangleq {1\over m_1m_2}\sum_{i,j}B_{i,j}$ denote the average value of the entries of $B$.

Let $\Cmatrix=\left(\Cmatrix_{ij}, i,j\geq 1\right)$ denote an infinite two dimensional array of independent standard normal random variables.
Denote by $\Cmatrix^{n \times m}$ the $n \times m$ upper left corner of $\Cmatrix$.
If $n = m$, we use $\Cmatrix^{n}$ instead.

The Large Average Submatrix algorithm is defined as follows.
\\
\begin{algorithmic}
\STATE \textbf{Large Average Submatrix algorithm ($\LAS$)} \\
\emph{
\STATE \textbf{Input}: An $n \times n$ matrix $A$ and a fixed integer $k\geq 1$.\\
\STATE \textbf{Initialize}: Select $k$ rows $I$ and $k$ columns $J$ arbitrarily.
\STATE \textbf{Loop}: (Iterate until no improvement is achieved) \\
$\quad$ Find the set $\hat J\subset [n], |\hat J|=k$ such that $\Ave(A_{I,\hat J})\ge \Ave(A_{I,J'})$ for all $J'\subset [n], |J'|=k$.
Break ties arbitrarily.\\
$\quad$ If $\hat J=J$, STOP. Otherwise,
set $J=\hat J$. \\
$\quad$
Find the set $\hat I\subset [n], |\hat I|=k$ such that $\Ave(A_{\hat I,J})\ge \Ave(A_{I',J})$ for all $I'\subset [n], |I'|=k$. Break ties arbitrarily.\\
$\quad$ If $\hat I=I$, STOP. Otherwise,
Set $I=\hat I$. \\
\STATE \textbf{Output}: $A_{I,J}$.
}
\end{algorithmic}
\bigskip
Since the entries of $\Cmatrix^n$ are continuous independent random variables the ties in the $\LAS$ algorithm occur with zero probability.
Each step of the $\LAS$ algorithm is easy to perform, since given a fixed set of rows $I$, finding the corresponding set of columns $\hat J$ which leads to the matrix with maximum
average entry is easy:
simply find $k$ columns corresponding to $k$ largest entry sums. Also the algorithm will stop after finitely many iterations since in each step the matrix sum (and the average)
increases and the number of submatrices is finite.
In fact a major part of our analysis is to bound the number of steps of $\LAS$. Our convention is that in iteration zero, the $\LAS$ algorithm
sets $I_0=I=\{1,\ldots,k\}$ and $J_0=J=\{1,\ldots,k\}$.
We denote by $T_{\LAS}$ the number of iterations of the $\LAS$ algorithm applied to the $n\times n$ matrix $\Cmatrix^n$ with i.i.d. standard normal entries.
For concreteness, searching for $\hat I$ and $\hat J$ are counted as two separate iterations. We denote by $\Cmatrix^n_r$ the matrix produced by
$\LAS$ in step (iteration) $r$, assuming $T_\LAS\ge r$.
Thus our goal is obtaining asymptotic values of $\Ave\left(\Cmatrix_{T_\LAS}\right)$, as well as the number of iterations $T_{\LAS}$.

Our first main result concerns the performance of $\LAS$ and stated  as follows. Let  $\omega_n$ denote any  positive function satisfying
$\omega_n=o(\sqrt{\log n})$ and $\log\log n=O(\omega_n)$.

\begin{theorem}\label{theorem:maintheorem}
Suppose a positive integer $k$ is fixed.
For every $\epsilon>0$ there is a positive integer $N$ which depends on $k$ and $\epsilon$ only,
such that for all $n \geq N$, $\pr(T_{\LAS}\ge N)\le \epsilon$. Furthermore,
\begin{align}
\label{eq:mainresult}
\lim_{n\rightarrow\infty}\mathbb{P}\left(\Big| \Ave(\Cmatrix_{T_{\LAS}}^n) - \sqrt{2\log n\over k} \Big|\leq \omega_n \right) =1.
\end{align}
\end{theorem}
Theorem \ref{theorem:maintheorem} states that the average of the $k \times k$ submatrix produced by $\LAS$ converges to the
value $(1+o(1))\sqrt{2\log n /k}$,
and furthermore, the number of iterations is stochastically bounded in $n$. In fact we will show the existence of a constant $0<\psi<1$ which depends
on $k$ and $\epsilon$ only such that $\pr(T_{\LAS}>t)\le \psi^t, t\ge 1$ . Namely, $T_{\LAS}$ is uniformly in $n$ bounded by a geometric random variable.

Next we turn to the performance of the greedy algorithm applied to the random graph produced from $\Cmatrix^n$ by first thresholding it at a certain level $\theta$.
Given $\Cmatrix^n$ let $\G(n,n,p(\theta))$ denote the corresponding $n\times n$ bi-partite graph where the edge $(i,j), i,j\in [n]$ is present if
$\Cmatrix^n_{i,j}>\theta$ and is absent otherwise. The edge probability is then $p(\theta)=\pr(Z>\theta)$ where $Z$ is a standard normal random variable.
A a pair of subsets $I\subset [n],J\subset [n]$ is a clique in $\G(n,n,p(\theta))$ if edge $(i,j)$ exists for every $i\in I,j\in J$. In this case
we write $i\sim j$.

Consider the following
simple algorithm for generating a clique in $\G(n,n,p(\theta))$, which we call greedy for simplicity. Pick node $i_1=1$ on the left part of the graph and
let $J_1=\{j:1\sim j\}$. Pick any node $j_1\in J_1$ and let $I_1=\{i\in [n]: i\sim j_1\}$. Clearly $i_1\in I_1$. Pick any node $i_2\in I_1$ different from $i_1$
and let $J_2=\{j\in J_1: i_2\sim j\}$. Clearly $j_1\in J_2$. Pick any $j_2\in J_2$ different from $j_1$ and let $I_2=\{i\in I_1: i\sim j_2\}$, and so on.
Repeat this process for as many steps $m$ as possible ending it on the right-hand side of the graph, so that the number of chosen nodes on the left
and the right is the same.
The end result $I_m,J_m$ is clearly a clique. It is also immediate that $|I_m|=|J_m|=m$.
The corresponding submatrix $\Cmatrix_{I_m,J_m}^n$ of $\Cmatrix^n$
indexed by rows $I_m$ and columns $J_m$ has every entry at least $\theta$ and therefore $\Ave(\Cmatrix_{I_m,J_m}^n)\ge \theta$. If we can guarantee that $\theta$ is small enough
so that $m$ is at least $k$, we obtain a simple algorithm for producing a $k\times k$ matrix with average entry at least $\theta$. From the theory of random graph it is known
(and easy to establish) that w.h.p. the greedy algorithm produces a clique of size $\log n/\log(1/p)$ provided that $p$ is at least $n^{-1+\epsilon}$ for some $\epsilon>0$.
Since we need to produce a $k\times k$ clique we obtain a requirement $\log n/\log(1/p)\ge k$ (provided of course the lower bound $n^{-1+\epsilon}$ holds, which
we will verify retroactively), leading to
\begin{align*}
p=\pr(Z>\theta)\ge n^{-{1\over k}},
\end{align*}
and in particular $k\ge 2$ is enough to satisfy the $n^{-1+\epsilon}$ lower bound requirement. Now suppose $k=o(\log n)$ implying $n^{-{1\over k}}=o(1)$.
The solving for $\theta_n$ defined by
\begin{align*}
\pr(Z>\theta_n)=n^{-{1\over k}}
\end{align*}
and using the fact
\begin{align*}
\lim_{t\rightarrow\infty} t^{-2}\log (Z>t)=-{1\over 2},
\end{align*}
we conclude that
\begin{align*}
\theta_n=(1+o(1))\sqrt{2\log n\over k},
\end{align*}
leading the same average value as the $\LAS$ algorithm! The two algorithms have asymptotically the same performance (though the greedy guarantees
a \emph{minimum} value of $(1+o(1))\sqrt{2\log n\over k}$ as opposed to just the (same) average value. We summarize our finding as follows.

\begin{theorem}\label{theorem:RandomGraphCliques}
Setting $\theta_n=(1+o(1))\sqrt{2\log n\over k}$, the greedy algorithm w.h.p. produces a $k\times k$ sub-matrix with minimum value
$\theta_n$ for $k=O(\log n)$.
\end{theorem}

Next we turn to an improved algorithm for finding a $k\times k$ submatrix with large average entry, which we call Incremental Greedy Procedure ($\IGP$)
and which achieves $(1+o_k(1))(4/3)\sqrt{2\log n/k}$  asymptotics.
We first provide a heuristic idea behind the algorithm which ignores certain dependencies, and then provide the appropriate fix for dealing
with the dependency issue.
The algorithm is described informally  as follows. Fix an arbitrary  $i_1\in [n]$
and in the corresponding row $\Cmatrix_{i_1,[n]}^n$  find the largest element $\Cmatrix_{i_1,j_1}^n$.
This term is asymptotically $\sqrt{2\log n}$ as the largest of $n$ i.i.d. standard normal random variables (see (\ref{eq:ExtremeGaussian1dim}) in Section~\ref{section:LAS}).
Then find the largest element $\Cmatrix_{i_2,j_1}^n$ in the column $\Cmatrix_{[n],j_1}$ other than $\Cmatrix_{i_1,j_1}$, which asymptotically is also   $\sqrt{2\log n}$.
Next in the $2\times n$ matrix $\Cmatrix_{\{i_1,i_2\},[n]}^n$ find a column $j_2\ne j_1$ such that the sum of the two elements of the column $\Cmatrix_{\{i_1,i_2\},j_2}^n$ is
larger than the sum for all other columns $\Cmatrix_{\{i_1,i_2\},j}^n$ for all $j\ne j_1$. Ignoring the dependencies, this sum is asymptotically $\sqrt{2}\sqrt{2\log n}$,
though the dependence is present here since the original row $\Cmatrix_{i_1,[n]}$ is a part of this computation.
We have created a $2\times 2$ matrix $\left(\Cmatrix_{i,j}^n, ~i=i_1,i_2;~j=j_1,j_2\right)$.
Then we find a row $i_3\ne i_1,i_2$ such that the sum of the two elements of the row  $\Cmatrix_{i_3,\{j_1,j_2\}}$ is larger than any other such sum of
$\Cmatrix_{i_3,\{j_1,j_2\}}$ for $i\ne i_1,i_2$. Again, ignoring the dependencies, this average is asymptotically $\sqrt{2}\sqrt{2\log n}$. We continue in this fashion,
greedily and incrementally expanding the matrix to a larger sizes,
creating in alternation $r \times r$ and $(r+1)\times r$ matrices and stop  when $r=k$ and we arrive at the $k\times k$ matrix.
In each step, ignoring the dependencies, the sum of the elements of the added row and added column is $\sqrt{r}\sqrt{2\log n}$ when the number of elements
in the row and in the column is $r$, again ignoring the dependency. Thus we expect the total asymptotic size of the final matrix to be
\begin{align*}
2\sum_{1\le r\le k-1}\sqrt{r}\sqrt{2\log n}+\sqrt{k}\sqrt{2\log n}.
\end{align*}
Approximating $2\sum_{1\le r\le k-1}\sqrt{r}+\sqrt{k}$ by $2\int_1^k \sqrt{x}dx \approx 4k^{3/2}/3 $ for growing $k$ and then dividing the expression above by $k^2$, we obtain the required asymptotics. The flaw in the argument above
comes from ignoring the dependencies: when $r\times 1$ row is chosen among the best such rows outside of the already created $r\times r$ matrix, the distribution
of this row is dependent on the distribution of this matrix. A simple fix comes from partitioning the entire $n\times n$ matrix into $k\times k$ equal size groups, and only searching
for the best $r\times 1$ row within the respective group. The sum of the elements of the $r$-th added row is then $\sqrt{r}\sqrt{2\log(n/k)}$ which is asymptotically
the same as  $\sqrt{r}\sqrt{2\log n}$, provided $k$ is small enough. The independence of entries between the groups is then used to estimate rigorously the performance
of the algorithm.

We now formalize the approach and state our main result. The proof or the performance of the algorithm is in Section~\ref{section:IGP}.
Given $n \in \mathbb{Z}^{+}$ and $k \in [n]$, divide the set $[n]$ into $k+1$ disjoint subsets, where the first $k$ subsets are
\begin{align*}
P_i^{n} = \{(i-1) \lfloor n/k \rfloor + 1, (i-1) \lfloor n/k \rfloor + 2, \ldots, i \lfloor n/k \rfloor \}, \text{ for } i=1,2,\ldots,k.
\end{align*}
When $n$ is a multiple of $k$, the last subset is by convention an empty set.
A detailed description of $\IGP$ algorithm is as follows.
\\
\begin{algorithmic}
\STATE $\IGP$ algorithm.\\
\emph{
\STATE \textbf{Input}: An $n \times n$ matrix $A$ and a fixed integer $k\geq 1$.
\STATE \textbf{Initialize}: Select $i_1\in P_1^n$  arbitrarily
and set $I=\{i_1\}$, and let $J=\emptyset$.
\STATE \textbf{Loop}: Proceed until $|I|=|J|=k$ \\
$\quad$ Find the column $j\in P_{|I|}^{n}$ such that $\Ave(A_{I,j}) \geq \Ave(A_{I,j'})$ for all $j' \in P_{|I|}^{n}$. Set $J = J \cup \{j\}$.\\
$\quad$
Find the  $i\in P_{|I|+1}^{n}$ such that $\Ave(A_{i,J}) \geq \Ave(A_{i',J})$ for all $i' \in P_{|I|+1}^{n}$. Set $I = I \cup \{i\}$.\\
\STATE \textbf{Output}: $A_{I,J}$.
}
\end{algorithmic}

\bigskip

As shown in Figure \ref{fig:step2ell}, $\mathcal{IGP}$ algorithm at step $2r$ adds a row of $r$ entries (represented by symbol `$\triangle$') with largest entry sum to the previous $r \times r$ submstrix $\mathbf{C}^{n,2r-1}_{\mathcal{IGP}}$. Similarly, as shown in Figure \ref{fig:step2ellplus1}, $\mathcal{IGP}$ algorithm at step $2r+1$ adds a column of $r+1$ entries (represented by symbol `$\triangle$') with largest entry sum to the previous $(r+1) \times r$ submstrix $\mathbf{C}^{n,2r}_{\mathcal{IGP}}$.

\begin{figure}[!ht]
  \centering
  \includegraphics[width=0.6\textwidth]{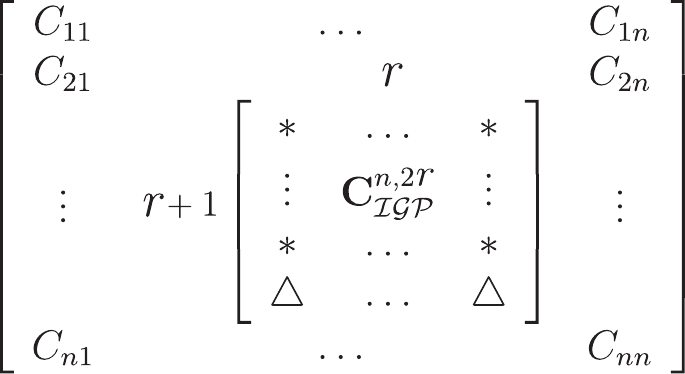}
  \caption{Step $2 r$ of $\IGP$ algorithm } \label{fig:step2ell}
\end{figure}

\begin{figure}[!htb]
  \centering
  \includegraphics[width=0.6\textwidth]{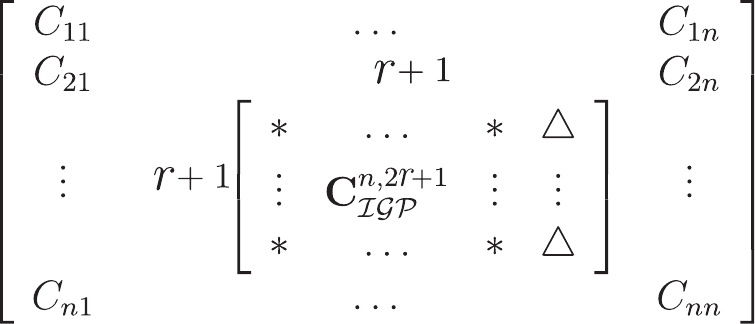}
  \caption{Step $2 r + 1$ of $\IGP$ algorithm} \label{fig:step2ellplus1}
\end{figure}

Just as for the $\LAS$ algorithm, each step of $\IGP$ algorithm is easy to perform: simply find one column (row) corresponding to the largest entry sum. The algorithm will stop after $2k$ steps. We denote by $\Cmatrix_{\IGP}^n$ the $k \times k$ submatrix produced by $\IGP$ applied to $\Cmatrix^n$. Our goal is to obtain the asymptotic value of $\Ave(\Cmatrix_{\IGP}^n)$.

Our main result regarding the performance of the $\IGP$ algorithm is as follows.
\begin{theorem}
\label{theorem:IGP_p}
Let $f(n)$ be any positive function such that $f(n)=o(n)$. Then
\begin{align}
\label{IGP_p_bound}
\lim_{n \rightarrow \infty} \min_{1\le k\le f(n)}\pr \left( \left\lvert \Ave(\Cmatrix_{\IGP}^n) - \frac{4}{3} \sqrt{\frac{2\log n}{k}} \right\rvert \leq
M \max\left( \frac{1}{k} \sqrt{\frac{ \log n}{k}} , \frac{\log \log n}{\sqrt{\log n}} \right)     \right) = 1.
\end{align}
\end{theorem}
The bound on the right hand side is of the order magnitude $O(\sqrt{\log n})$ when $k$ is constant and $o(\sqrt{\log n/k})$ when $k$ is a growing function of $n$.
The asymptotics $(1+o_k(1))\frac{4}{3} \sqrt{\frac{2\log n}{k}}$  corresponds to the latter case. Also, while the theorem is valid for $k\le f(n)=o(n)$,
it is only interesting for $k=o(\log^2/(\log\log n)^2)$, since otherwise the error term $\frac{\log \log n}{\sqrt{\log n}}$ is comparable with the
value $\frac{4}{3} \sqrt{\frac{2\log n}{k}}$.

Next we turn to the discussion of the Overlap Gap Property (OGP).
Fix $\alpha\in (1,\sqrt{2})$,  real values $0\le y_1,y_2\le 1$ and $\delta>0$.
Let $\mathcal{O}(\alpha,y_1,y_2,\delta)$ denote the set of pairs of $k\times k$ submatrices $\Cmatrix_{I_1,J_1}^n, \Cmatrix_{I_2,J_2}^n$ with average value
in the interval $[(\alpha-\delta)\sqrt{2\log n/k},(\alpha+\delta)\sqrt{2\log n/k}]$ and which satisfy
$|I_1\cap I_2|/k\in (y_1-\delta,y_1+\delta), |J_1\cap J_2|/k\in (y_2-\delta,y_2+\delta)$. Namely, $\mathcal{O}(\alpha,y_1,y_2,\delta)$
is the set of pairs of $k\times k$ matrices with average value approximately $\alpha\sqrt{2\log n/k}$ and which share approximately
$y_1k$ rows and $y_2k$ columns. Let
\begin{align}
f(\alpha,y_1,y_2)\triangleq 4-y_1-y_2-{2\over 1+y_1y_2}\alpha^2. \label{eq:f}
\end{align}
The next result says that the expected cardinality of the set $\mathcal{O}(\alpha,y_1,y_2,\delta)$ is approximately $n^{kf(\alpha,y_1,y_2)}$ when
$f(\alpha,y_1,y_2)$ is positive, and, on the other hand, $\mathcal{O}(\alpha,y_1,y_2,\delta)$ is empty with high probability when $f(\alpha,y_1,y_2)$ is negative.

\begin{theorem}\label{theorem:Number of pairs}
For every $\epsilon>0$ and $c>0$, there exists $\delta>0$ and $n_0>0$ such that for all $n\ge n_0$ and $k \le c\log n$
\begin{align}
\label{eq:ConvergeExpectation}
\left| \frac{\log \E\left[\big|\mathcal{O}(\alpha,y_1,y_2,\delta)\big|\right]}{k \log n}-f(\alpha,y_1,y_2) \right| < \epsilon.
\end{align}
As a result, when $f(\alpha,y_1,y_2)<0$, for every $\epsilon>0$ and $c>0$, there exists $\delta>0$ and $n_0>0$ such that for all $n\ge n_0$ and $k \le c\log n$
\begin{align}
\label{eq:ProbabilityBound}
\pr\left(\mathcal{O}(\alpha,y_1,y_2,\delta)\ne \emptyset\right)<\epsilon.
\end{align}
\end{theorem}
We see that the region $\mathcal{R}(\alpha)\triangleq
\{(y_1,y_2): f(\alpha,y_1,y_2) \geq 0\}$ identifies the region of achievable in expectation overlaps for matrices with average values
approximately $\alpha\sqrt{2\log n/k}$.

Regarding $\mathcal{R}(\alpha)$, we establish two phase transition points: one at $\alpha_1^* = \sqrt{3/2}$ and the other one at $\alpha_2^* = 5\sqrt{2}/(3\sqrt{3})$. The derivation of these values is delayed till Section \ref{section:OGP}. Computing $\mathcal{R}(\alpha)$ numerically we see that it exhibits three qualitatively different behaviors for $\alpha \in (0,\alpha_1^*)$, $(\alpha_1^*, \alpha_2^*)$ and $(\alpha_2^*, \sqrt{2})$, respectively, as shown in Figures \ref{fig:figure0}, \ref{fig:figure1} and \ref{fig:figure3}.

\begin{enumerate}
\item[(a)] When $\alpha\in (1,\sqrt{3}/\sqrt{2})$, $\mathcal{R}(\alpha)$ coincides with the entire region $[0,1]^2$, see Figure \ref{fig:figure0}. From the heat map of the figure,
with dark color corresponding to the higher value of $f$ and light color corresponding to the lower value, we also see that the bulk of the overlap corresponds to values
of $y_1,y_2$ which are close to zero. In other words, the picture suggests that most matrices with average value approximately
$\alpha\sqrt{2\log n/k}$ tend to be far from each other.

\item[(b)] When $\alpha\in (\sqrt{3}/\sqrt{2}, 5\sqrt{2}/(3\sqrt{3}))$, we see that $\mathcal{R}(\alpha)$ is a connected subset of  $[0,1]^2$, (Figure \ref{fig:figure1}), but a non-achievable overlap region emerges
(colored white on the figure) for pairs of matrices with this average value. At a critical value $\alpha = 5\sqrt{2}/(3\sqrt{3})$ the set is connected through a single point $(1/3,1/3)$, see Figure \ref{fig:figure2}.

\item[(c)] When $\alpha\in (5\sqrt{2}/(3\sqrt{3}),\sqrt{2})$, $\mathcal{R}(\alpha)$ is a disconnected subset of $[0,1]^2$ and the OGP emerges, see Figure \ref{fig:figure3} for
$\alpha = 1.364$. In this case,
every pair of matrices
has either approximately at least $0.4k$ common columns or at most $0.28k$ common columns.
\end{enumerate}

\begin{figure}[!htb]
\minipage{0.5\textwidth}
  \includegraphics[width=\linewidth]{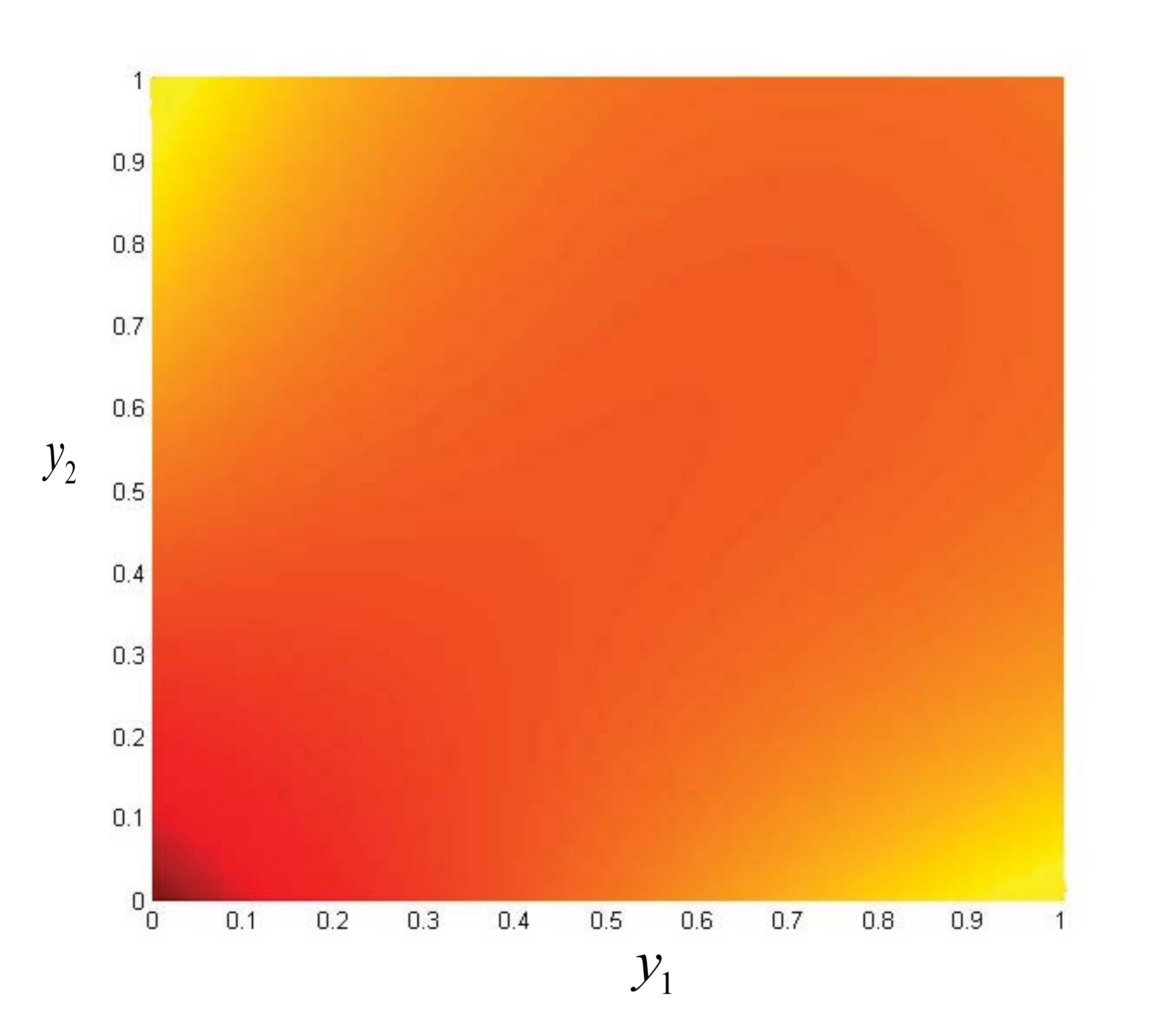}
  \caption{ $\mathcal{R}(\alpha)$ for  $\alpha \in (0, \sqrt{3}/\sqrt{2})$ } \label{fig:figure0}
\endminipage\hfill
\minipage{0.5\textwidth}
  \includegraphics[width=\linewidth]{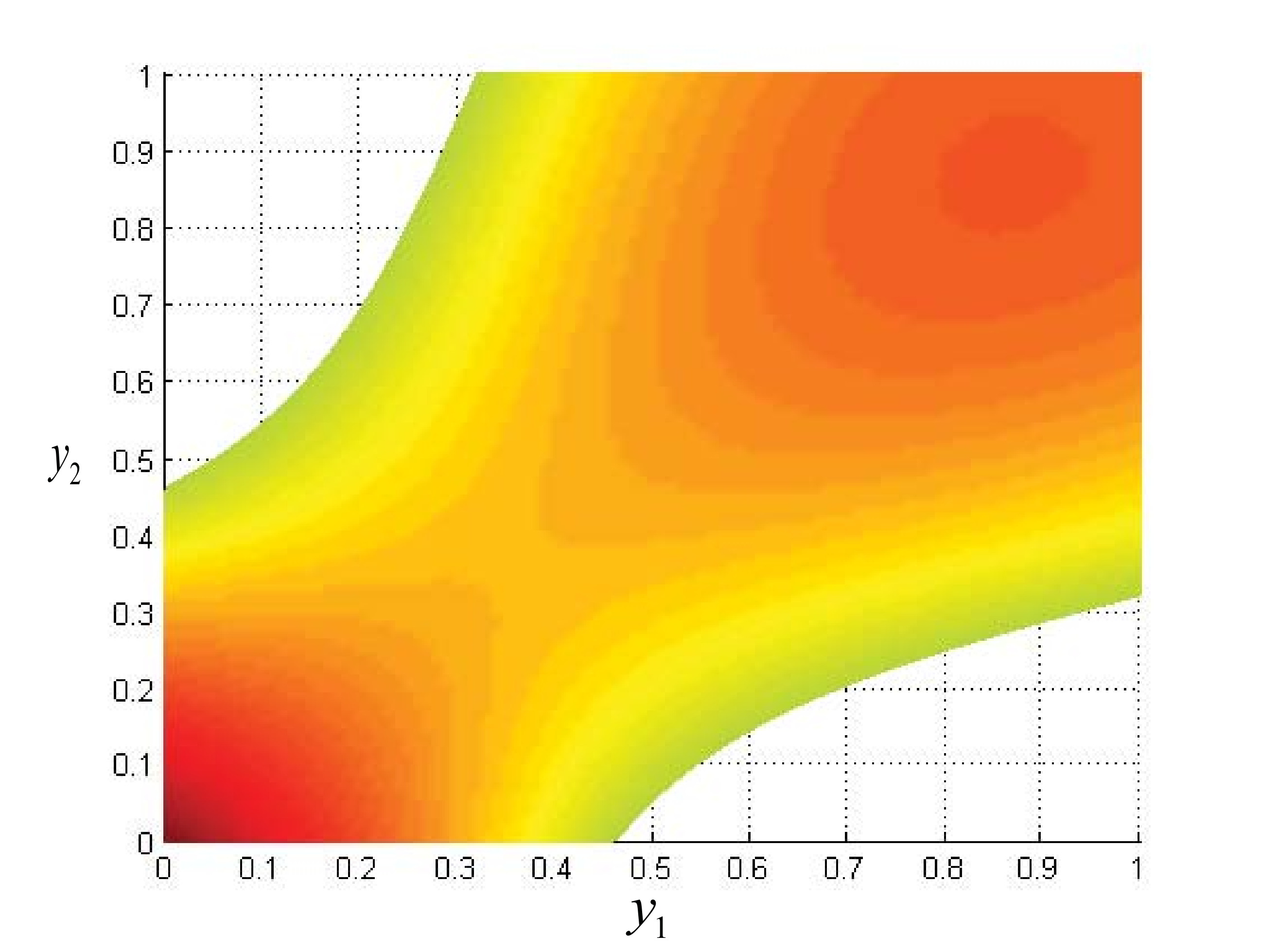}
  \caption{ $\mathcal{R}(\alpha)$ for  $\alpha \in (\sqrt{3}/\sqrt{2}, 5\sqrt{2}/(3\sqrt{3}))$ } \label{fig:figure1}
\endminipage\hfill
\minipage{0.5\textwidth}
  \includegraphics[width=\linewidth]{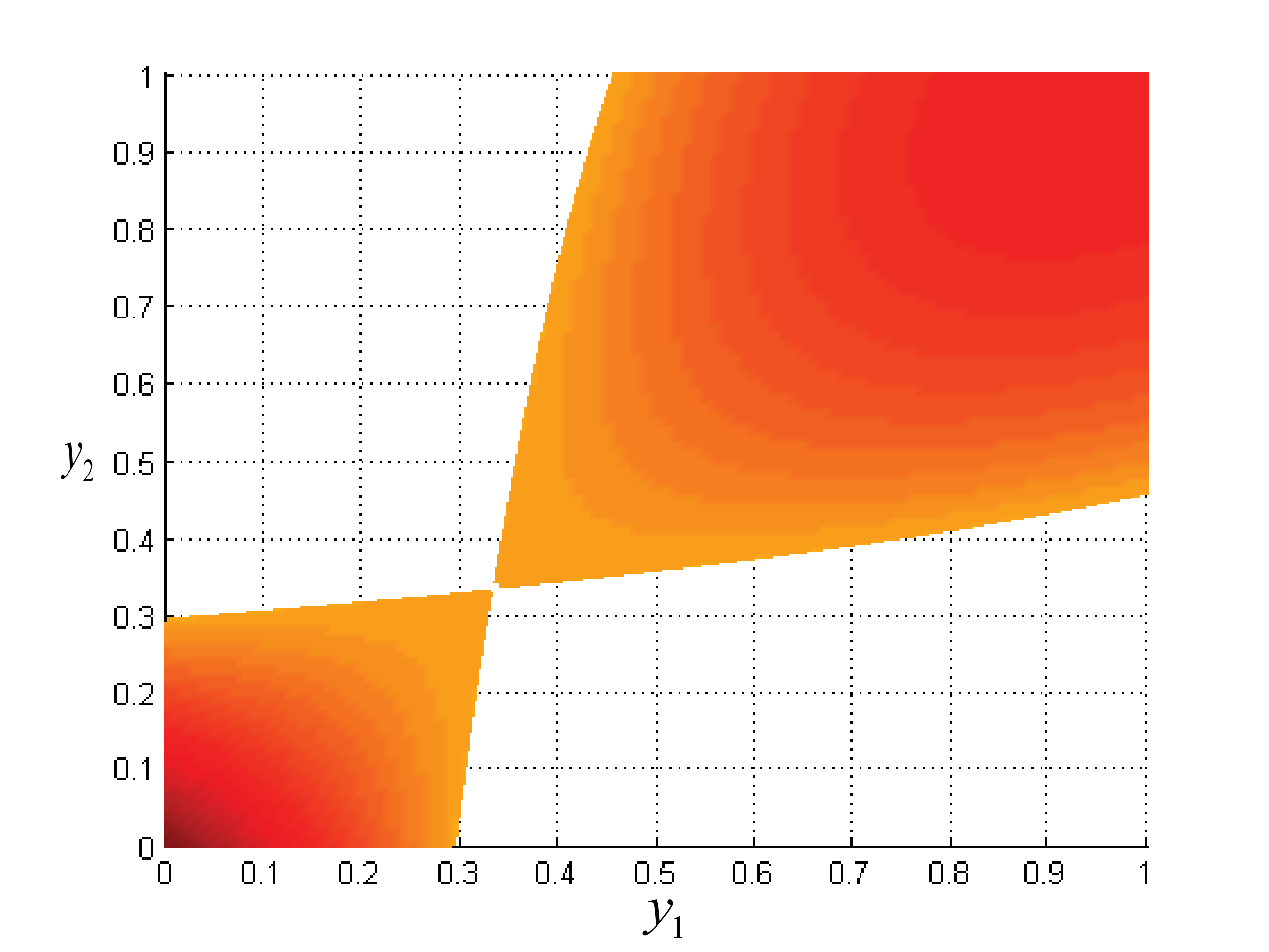}
  \caption{ $\mathcal{R}(5\sqrt{2}/(3\sqrt{3}))$ }  \label{fig:figure2}
\endminipage\hfill
\minipage{0.5\textwidth}%
  \includegraphics[width=\linewidth]{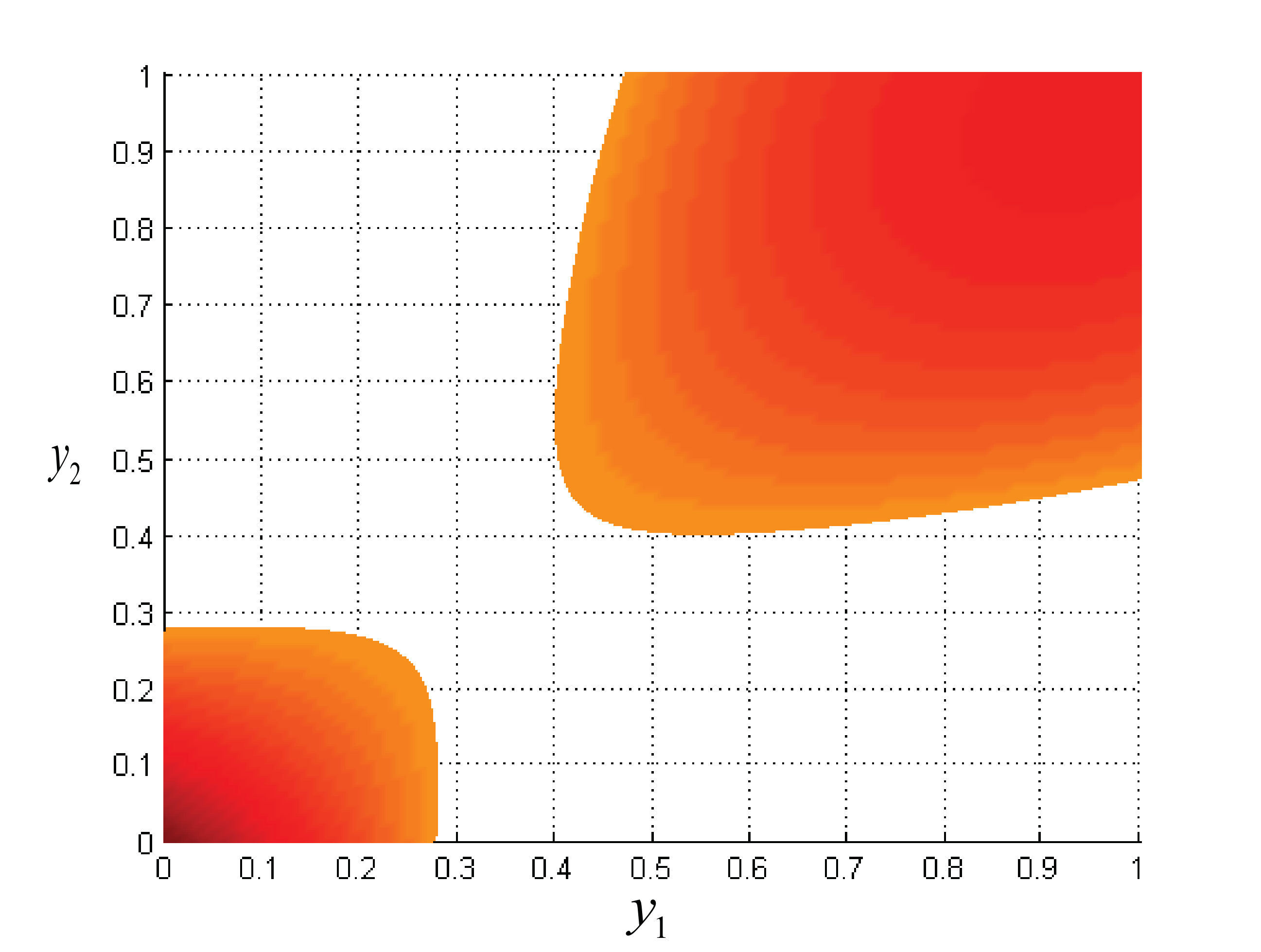}
  \caption{$\mathcal{R}(\alpha)$ for  $\alpha\in (5\sqrt{2}/(3\sqrt{3}), \sqrt{2})$} \label{fig:figure3}
\endminipage

\end{figure}

We conjecture that the regime (c) described on Figure \ref{fig:figure3} corresponds to the hard on average case for which we predict
that no polynomial time algorithm exists for non-constant  $k$.
Since the OGP was analyzed based on overlaps of two matrices and the overlap of three matrices is likely to push the critical
value of OGP even lower, we conjecture that the hardness regime begins at a value lower than our current estimate $5\sqrt{2}/(3\sqrt{3})$. An interesting open question is to conduct
an overlap analysis of $m$-tuples of matrices and identify the critical value for the onset of disconnectedness.

\section{Analysis of the $\IGP$ algorithm}\label{section:IGP}
This section is devoted to the  proof of Theorem~\ref{theorem:IGP_p}.
Denote by $I_r^n$ the set of rows produced by $\IGP$ algorithm in steps $2r$, $r=0,1,\ldots, k-1$ and by $J_r^n$ the set of columns produced by $\IGP$ algorithm in steps $2r-1$, $r=1,\ldots, k$. Their cardinalities satisfy $|I_{r}^n|=r+1$ for $r=0,1,\ldots,k-1$ and $|J_r^n| = r$ for $r=1,\ldots,k$.
In particular, $\IGP$ algorithm chooses $I_0^n = \{i_1\}$ arbitrarily from $P_1^n$ and $J_1^n$ is obtained by finding the column in $\Cmatrix_{i_1,P_1^{n}}$ corresponding to the largest entry.
Let $M_i^{n}$, $i=1,2,\cdots,2k-1$ be the entry sum of the row or column $\IGP$ algorithm adds to the submatrix in the $i$-th step, namely
\begin{align}
\label{def_Mnt}
& M_{2r-1}^{n} \triangleq \max_{j\in P_{|I_{r-1}^{n}|}^{n}} \sum_{i \in I_{r-1}^{n}} C_{i,j}  \text{ for } r=1,2,\ldots,k,  \nonumber \\
& M_{2r}^{n} \triangleq \max_{i \in P_{|J_{r}^{n}|+1}^{n}} \sum_{j \in J_{r}^{n}} C_{i,j} \text{ for } r=1,2,\ldots,k-1.
\end{align}

Introduce
\begin{align}
\label{eq:bn}
b_n:=\sqrt{2\log n}-\frac{\log(4\pi \log n)}{2\sqrt{2\log n}}.
\end{align}
In order to quantify $M_i^{n}$, $i=1,2,\ldots,2k-1$, we now introduce a probabilistic bound on the maximum of $n$ independent standard normal random variables.

\begin{lemma}
\label{extreme_Gaussian}
Let $Z_i$, $i=1,2,\ldots,n$ be $n$ independent i.i.d. standard normal random variables. There exists a positive integer $N$ such that for all $n >N$
\begin{align}
\label{dist_fun}
\pr\left( \left\lvert  \sqrt{2 \log n} \left(\max_{1\leq i \leq n} Z_i-b_n \right) \right\rvert \leq  \log \log n \right)   \geq 1-\frac{1}{(\log n)^{1.4}}.
\end{align}
\end{lemma}

Lemma \ref{extreme_Gaussian} is a cruder version of the well-known fact described later in Section \ref{section:LAS} as fact (\ref{eq:ExtremeGaussian1dim}).
For convenience, in what follows, we use $n/k$ in place of $\lfloor n/k \rfloor$. We first establish Theorem~\ref{theorem:IGP_p} from the lemma above, the proof of which we delay for later.

\begin{proof}[Proof of Theorem~\ref{theorem:IGP_p}]
Denote by $E_{2r-1}^{n}$, $r=1,2,\ldots,k$ the event that
\begin{align}
\label{def_Enq1}
\left\lvert \sqrt{2\log (n/k)} \left(\frac{M_{2r-1}^{n}}{\sqrt{\lvert I_{r-1}^{n} \rvert}} -b_{n/k} \right) \right\rvert \leq \log\log (n/k),
\end{align}
and by $E_{2r}^{n}$, $r=1,2,\ldots,k-1$ the event that
\begin{align}
\label{def_Enq2}
\left\lvert \sqrt{2\log (n/k)} \left(\frac{M_{2r}^{n}}{ \sqrt{  \lvert J_{r}^{n} \rvert}} -b_{n/k} \right) \right\rvert \leq \log\log (n/k).
\end{align}
By Lemma \ref{extreme_Gaussian} and since $k \leq f(n)=o(n)$, we can choose a positive integer $N_1$ such that for all $n > N_1$
\begin{align}
\label{pro_ineq_Enq1}
\pr \left( E_i^{n} \right) \geq 1-\frac{1}{(\log (n/k))^{1.4}}, \qquad \forall~1\leq i \leq 2k-1.
\end{align}
Since $M_i^{n}$, $i=1,2,\cdots,2k-1$ corresponds to non-overlapping parts of $\Cmatrix^n$, they are mutually independent, and so are $E_i^{n}$, $i=1,2,\cdots,2k-1$. Choose another positive integer $N_2$ such that for all $n > N_2$,
$$
\frac{1}{(\log n)^{0.3}} \geq \frac{1}{(\log (n/k))^{1.4}}(2k-1).
$$
Let $N \triangleq \max(N_1, N_2)$. Then for all $n > N$ we have
\begin{align*}
\pr \left(\cap_{i=1}^{2k-1} E_i^{n} \right) = \prod_{i=1}^{2k-1} \mathbb{P}(E_i^{n}) & \geq \left(1-\frac{1}{(\log (n/k))^{1.4}} \right)^{2k-1}  \\
                                                                                   & \geq 1 - \frac{1}{(\log (n/k))^{1.4}}(2k-1) \geq 1 - \frac{1}{(\log n)^{0.3}}.
\end{align*}
As a result, $\cap_{i=1}^{2k-1} E_i^{n}$ occurs w.h.p..

We can choose a positive integer $N_3$ such that for all $n > N_3$ and $k \leq f(n)=o(n)$, $2\log (n/k) \geq \log n$ holds. Then under the event $\cap_{i=1}^{2k-1} E_i^{n}$ and for all $n > N_3$, we use (\ref{def_Enq1}) and (\ref{def_Enq2}) to estimate the average value of $\Cmatrix_{\IGP}^n$
\begin{align}
 \Ave(\Cmatrix_{\IGP}^n) & \leq \frac{1}{k^2} \Bigg( \sum_{r=1}^k \left( \sqrt{\lvert I_{r-1}^{n}\rvert} b_{n/k} +  \sqrt{\lvert I_{r-1}^{n} \rvert} \frac{\log \log (n/k)}{\sqrt{2 \log (n/k)}} \right)  \nonumber \\
                   & \quad \quad \quad +  \sum_{r=1}^{k-1} \left( \sqrt{\lvert J_{r}^{n}\rvert} b_{n/k} +  \sqrt{\lvert J_{r}^{n} \rvert} \frac{\log \log (n/k)}{\sqrt{2 \log (n/k)}}  \right)  \Bigg)                                                                           \nonumber \\
& \leq \frac{\sum_{i=1}^k \sqrt{2\log n} \sqrt{i} + \sum_{i=1}^{k-1} \sqrt{2\log n} \sqrt{i} }{k^2} + \frac{2\log \log n}{\sqrt{\log n}}   \nonumber \\
& = 2 \sqrt{\frac{2 \log n}{k}} \sum_{i=1}^k \sqrt{\frac{i}{k}}\frac{1}{k} -\frac{\sqrt{2\log n}}{k^{3/2}} + \frac{2\log \log n}{\sqrt{\log n}}   \nonumber \\
\label{eq:MIntroduce}
& \leq 2 \sqrt{\frac{2 \log n}{k}} \int_{0}^1 \sqrt{x} dx + \max\left(\frac{1}{k} \sqrt{\frac{ \log n}{k}} , \frac{\log \log n}{\sqrt{\log n}} \right)        \\
& =  \frac{4}{3} \sqrt{\frac{2 \log n}{k}}+  2 \max \left(\frac{1}{k} \sqrt{\frac{ \log n}{k}} , \frac{\log \log n}{\sqrt{\log n}} \right) \nonumber
\end{align}
Similarly we can show
$$
 \Ave(\Cmatrix_{\IGP}^n)  \geq \frac{4}{3} \sqrt{\frac{2 \log n}{k}} -  2 \max \left(\frac{1}{k} \sqrt{\frac{ \log n}{k}} , \frac{\log \log n}{\sqrt{\log n}} \right).
$$
Then (\ref{IGP_p_bound}) follows and the proof is completed.
\end{proof}

We now return to the proof of Lemma \ref{extreme_Gaussian}. Let $\Phi(u)$ be the cumulative distribution function of the standard normal random variable. When $u$ is large, the function $1-\Phi(u)$ can be approximated by
\begin{align}
\label{approx_Phi_u}
\frac{1}{u\sqrt{2\pi}} \exp(-u^2/2)(1-2 u^{-2}) \leq 1 - \Phi(u) \leq \frac{1}{u\sqrt{2\pi}} \exp(-u^2/2).
\end{align}
Recall that  $\omega_n$ denotes any strictly increasing positive function satisfying
$\omega_n=o(\sqrt{2\log n})$ and $\log\log n = O(\omega_n)$.
\begin{proof}[Proof of Lemma \ref{extreme_Gaussian}]
We have
\begin{align}
\label{pro_bound}
& \pr\left( \left\lvert  \sqrt{2 \log n} \left(\max_{1\leq i \leq n} Z_i-b_n \right) \right\rvert \leq  \log \log n \right) \nonumber \\
& = \pr \left(\max_{1\leq i \leq n} Z_i \leq  \log \log n /\sqrt{2 \log n} +b_n  \right) - \pr \left(\max_{1\leq i \leq n} Z_i<-\log \log n /\sqrt{2 \log n} +b_n  \right) \nonumber  \\
& = \pr \left(Z_1 \leq  \log \log n/\sqrt{2 \log n} + b_n \right)^n - \pr \left(Z_1 < -\log \log n/\sqrt{2 \log n} + b_n \right)^n
\end{align}
Next, we use (\ref{approx_Phi_u}) to approximate
\begin{align}
\label{first_term}
& \pr \left(Z_1 \leq  \log \log n/\sqrt{2 \log n} + b_n \right)  \nonumber \\
& = 1 - (1+o(1)) \frac{1}{(  \log \log n / \sqrt{2 \log n} + b_n)\sqrt{2\pi}}\exp\left( -\frac{(  \log \log n / \sqrt{2 \log n}+b_n)^2}{2} \right)  \nonumber \\
& = 1 - \Theta \left( \frac{1}{n (\log n)^{3/2}} \right)
\end{align}
and
\begin{align}
\label{second_term}
& \pr \left(Z_1 < -\log \log n/\sqrt{2 \log n} + b_n \right)  \nonumber \\
& = 1 - (1+o(1)) \frac{1}{( - \log \log n / \sqrt{2 \log n} + b_n)\sqrt{2\pi}}\exp\left( -\frac{( - \log \log n / \sqrt{2 \log n}+b_n)^2}{2} \right)  \nonumber \\
& = 1- \Theta \left(\frac{\sqrt{\log n}}{n}  \right).
\end{align}
Now we substitute (\ref{first_term}) and (\ref{second_term}) into (\ref{pro_bound})
\begin{align*}
& \pr\left( \left\lvert  \sqrt{2 \log n} \left(\max_{1\leq i \leq n} Z_i-b_n \right) \right\rvert \leq  \log \log n \right)   \\
& = \left( 1- \Theta \left( \frac{1}{n (\log n)^{3/2}} \right)  \right)^n - \left(1-\Theta \left(\frac{\sqrt{\log n}}{n} \right) \right)^n \\
& = \left( 1- \Theta \left( \frac{1}{n (\log n)^{3/2}} \right)  \right)^n - \exp(-\Theta(\sqrt{\log n})).
\end{align*}
Then the result follows from choosing a positive integer $N$ such that for all $n >N$ the following inequality holds
$$
 \left( 1- \Theta \left( \frac{1}{n (\log n)^{3/2}} \right)  \right)^n - \exp(-\Theta(\sqrt{\log n}))  \geq 1  - \frac{1}{(\log n)^{1.4}}.
$$
\end{proof}

\section{The Ovelap Gap Property}\label{section:OGP}

In this section, we first derive the critical values for the two phase transition points
$\alpha_1^* = \sqrt{3}/\sqrt{2}$ and $\alpha_2^* = 5 \sqrt{2}/(3 \sqrt{3})$ and then complete the proof of Theorem \ref{theorem:Number of pairs}.

We start with $\alpha_1^*$ which we define as a critical point such that for any $\alpha>\alpha_1^*$ and $\alpha \in (0,\sqrt{2})$, $\mathcal{R}(\alpha)$ does not cover the whole region $[0,1]^2$, i.e. $[0,1]^2 \setminus \mathcal{R}(\alpha) \neq \emptyset$. We formulate this as follows
\begin{align}
\label{eq: optimizef}
\alpha_1^* \triangleq \max\{\alpha \in (0,\sqrt{2}): \min_{y_1,y_2 \in [0,1]^2} f(\alpha,y_1,y_2) \geq 0\}.
\end{align}
Since $f(\alpha,y_1,y_2)$ is differentiable with respect to $y_1$ and $y_2$, the minimum of $f(\alpha,y_1,y_2)$ for a fixed $\alpha$ appear either at the boundaries or the stationary points. Using the symmetry of $y_1$ and $y_2$, we only need to consider the following boundaries
$$
\{(y_1,y_2): y_1=0, y_2 \in [0,1]\} \cup \{(y_1,y_2): y_1=1, y_2 \in [0,1]\}.
$$
By inspection, $\min_{y_1=0, y_2 \in [0,1]} f(\alpha, y_1,y_2)=3-2\alpha^2$ and
\begin{align*}
\min_{y_1=1, y_2 \in [0,1]} f(\alpha, y_1,y_2) = \min_{y_2 \in [0,1]} \left\{ 3-y_2-\frac{2}{1+y_2}\alpha^2  \right\}.
\end{align*}
Since the objective function above is a concave function with respect to $y_2$, its minimum is obtained at $y_2=0 \text{ or } 1$, which is $3-2\alpha^2$ or $2-\alpha^2$. Hence the minimum of $f(\alpha, y_1, y_2)$ at the boundaries above is either $3-2\alpha^2$ or $2-\alpha^2$. Both of them being nonnegative requires
$$
3-2\alpha^2 \geq 0  \text{ and } 2-\alpha^2 \geq 0 \text{ and } \alpha  \in (0,\sqrt{2})  \Rightarrow  \alpha \in (0, \sqrt{3}/\sqrt{2}].
$$
Next we consider the stationary points of $f(\alpha,y_1,y_2)$ for a fixed $\alpha$.
The stationary points are determined by solving
\begin{align*}
\frac{\partial f(\alpha, y_1, y_2)}{\partial y_1}=0 \Rightarrow  -1 + \frac{2 \alpha^2 y_2}{(1+y_1y_2)^2} = 0  \\
\frac{\partial f(\alpha, y_1, y_2)}{\partial y_2}=0 \Rightarrow  -1 + \frac{2 \alpha^2 y_1}{(1+y_1y_2)^2} = 0
\end{align*}
Observe from above $y_1 = y_2$. Then we can simplify the equations above by
\begin{align}
\label{stationarypoint}
y_1^4 + 2 y_1^2 - 2\alpha^2 y_1 + 1 = 0
\end{align}
Using 'Mathematica', we find that the four solutions for the quartic equation above for $\alpha^2 = 3/2$ are complex numbers all with nonzero imaginary parts. Since the equation above does not have real solutions, the optimization problem (\ref{eq: optimizef}) has maximum at $\alpha = \sqrt{3}/\sqrt{2}$. On the other hand, for any $\alpha > \sqrt{3}/\sqrt{2}$, $f(\alpha, 1, 0)=3-2\alpha^2$ is always negative. Hence, we have $\alpha_1^* = \sqrt{3}/\sqrt{2}$.

We also claim that for any $\alpha \in (0,\sqrt{3}/\sqrt{2})$, $\mathcal{R}(\alpha) = [0,1]^2$. It suffices to show that for any $y \in [0,1]$,
$$
y^4 + 2 y^2 - 2\alpha^2 y + 1 > 0.
$$
Suppose there is a $\hat{y} \in [0,1]$ such that $\hat{y}^4 + 2 \hat{y}^2 - 2\alpha^2 \hat{y} + 1 \leq  0$. Then by $\alpha^2 < 3/2$ and $\hat{y} \neq 0$ we have
$$
\hat{y}^4 + 2 \hat{y}^2 - 3 \hat{y} + 1 < 0.
$$
Since $y^4 + 2 y^2 - 3 y + 1$ is positive at $y=0$ and negative at $\hat{y}$, the continuity of $y^4 + 2 y^2 - 3 y + 1$ implies that there is a $y_1 \in [0,1]$ such that (\ref{stationarypoint}) holds for $\alpha^2 = 3/2$, which is a contradiction. The claim follows.

Next we introduce $\alpha_2^*$. Increasing $\alpha$ beyond $\alpha_1^*$, we are interested in the first point $\alpha_2^*$ at which the function $f(\alpha_2^*,y_1,y_2)$ has at least one real stationary point and the value of $f(\alpha_2^*,y_1,y_2)$ at this point is zero. Observe that at the stationary points $y_1 = y_2$ and $y_1$ satisfies (\ref{stationarypoint}). Then $\alpha_2^*$ is determined by solving
\begin{align*}
& y_1^4 + 2 y_1^2 - 2\alpha^2 y_1 + 1 = 0, \\
& 4-2y_1-\frac{2}{1+y_1^2} \alpha^2 = 0, \\
& y_1 \in [0,1], \quad \alpha \in (\sqrt{3}/\sqrt{2}, \sqrt{2}).
\end{align*}
Using `mathematica' to solve the equations above, we obtain only one real solution $y_1=1/3, \alpha=5\sqrt{2}/(3\sqrt{3})$. Then we have $\alpha_2^* = 5\sqrt{2}/(3\sqrt{3})$ and $f(\alpha_2^*,1/3,1/3)=0$. We verify that $f(\alpha_2^*,1/3,y_2) < 0$ for $y_2 \in [0,1] \setminus \{1/3\}$ and $f(\alpha_2^*,y_1,1/3) < 0$ for $y_1 \in [0,1] \setminus \{1/3\}$. By plotting $f(\alpha_2^*,y_1,y_2)$ in Figure \ref{fig:figure2}, we see that the set $\mathcal{R}(\alpha_2^*)$ is connected through a single point $(1/3,1/3)$.

\begin{proof}[Proof of Theorem \ref{theorem:Number of pairs}]
The rest of the section is devoted to part (\ref{eq:ConvergeExpectation}) of Theorem \ref{theorem:Number of pairs}. The second result (\ref{eq:ProbabilityBound}) follows from the Markov inequality.

Fix positive integers $k_1$, $k_2$, $k$ and $n$ such that $k_1 \leq k \leq n$ and $k_2 \leq k \leq n$. Let $X$, $Y_1$ and $Y_2$ be three mutually independent normal random variables: $X \,{\buildrel d \over =}\, \mathcal{N}(0,k_1k_2)$ and $Y_1 \,{\buildrel d \over =}\, Y_2 \,{\buildrel d \over =}\  \mathcal{N}(0,k^2-k_1k_2)$. Then
\begin{align}
\label{exp_pair}
& \mathbb{E}(\lvert \mathcal{O}(\alpha,y_1,y_2, \delta)\rvert) \nonumber \\
=&  \sum_{\substack{ k_1 \in ((y_1-\delta)k, (y_1+\delta)k) \\
                                                                                  k_2 \in ((y_2-\delta)k, (y_2+\delta)k)}}{n \choose {k-k_1,k_1,k-k_1}} {n \choose {k-k_2,k_2,k-k_2}} \times   \nonumber \\
                                                           & \times \mathbb{P} \left(X+Y_1, X+Y_2 \in \left[(\alpha-\delta)k^2
																													\sqrt{\frac{2\log n}{k}}, (\alpha+\delta)k^2
																													\sqrt{\frac{2\log n}{k}} \right] \right).
\end{align}
First, we estimate the last term in (\ref{exp_pair}). For the special case $k_1 = k_2 = k$, observing $Y_1 = Y_2 =0$ and using (\ref{approx_Phi_u}) we obtain
\begin{align*}
& \frac{1}{k \log n}  \log \mathbb{P} \left(X+Y_1, X+Y_2 \in \left[(\alpha-\delta)k^2
																													\sqrt{\frac{2\log n}{k}}, (\alpha+\delta)k^2
																													\sqrt{\frac{2\log n}{k}} \right] \right)  \\
&=  \frac{1}{k \log n} \log \mathbb{P} \left(X \in \left[(\alpha-\delta)k^2
																													\sqrt{\frac{2\log n}{k}}, (\alpha+\delta)k^2
																													\sqrt{\frac{2\log n}{k}} \right] \right) = o(1) - (\alpha-\delta)^2.
\end{align*}
This estimate will be used later. Now we consider the case where at lease one of $k_1$ and $k_2$ is smaller than $k$. We let $\tau \triangleq (\alpha-\delta)\sqrt{2k_1k_2/(k^2+k_1k_2)}$ and write
\begin{align*}
 \mathbb{P} \left( X+Y_1, X+Y_2 \in \left[(\alpha-\delta)k^2
																													\sqrt{\frac{2\log n}{k}}, (\alpha+\delta)k^2
																													\sqrt{\frac{2\log n}{k}} \right] \right) =   I_1 + I_2
\end{align*}
where
{\small
\begin{align*}
& I_1 = \int_{-\infty}^{\tau k^2 \sqrt{\frac{2\log n}{k}} } \mathbb{P}\left(  (\alpha+\delta)k^2
																													\sqrt{\frac{2\log n}{k}}-x \geq Y_1 \geq (\alpha-\delta)k^2
																													\sqrt{\frac{2\log n}{k}}-x \right)^2  \frac{1}{\sqrt{2\pi k_1k_2}}\exp \left(-\frac{x^2}{2k_1k_2} \right)dx, \\
& I_2 = \int_{\tau k^2 \sqrt{\frac{2\log n}{k}} }^{\infty} \mathbb{P}\left(  (\alpha+\delta)k^2
																													\sqrt{\frac{2\log n}{k}}-x \geq Y_1 \geq (\alpha-\delta)k^2
																													\sqrt{\frac{2\log n}{k}}-x \right)^2 \frac{1}{\sqrt{2\pi k_1k_2}} \exp \left(-\frac{x^2}{2k_1k_2} \right) dx.
\end{align*}
}
In order to use (\ref{approx_Phi_u}) to approximate the integrand in $I_1$, we need to verify that for $x \leq  \tau k^2 \sqrt{\frac{2\log n}{k}}$, the following quantity goes to infinity as $n \rightarrow \infty$:
\begin{align*}
\frac{(\alpha-\delta)k^2 \sqrt{\frac{2\log n}{k}}-x}{\sqrt{k^2-k_1 k_2}} & \geq \frac{(\alpha-\delta-\tau) k^2 \sqrt{\frac{2\log n}{k}}}{\sqrt{k^2-k_1 k_2}} \\
                                              & = \frac{1-\sqrt{2k_1k_2/(k^2+k_1k_2)}}{\sqrt{k^2-k_1 k_2}}(\alpha-\delta) k^2 \sqrt{\frac{2\log n}{k}} \\
																							& = \frac{1-\sqrt{1-(k^2-k_1k_2)/(k^2+k_1k_2)}}{\sqrt{k^2-k_1 k_2}}   (\alpha-\delta) k^2 \sqrt{\frac{2\log n}{k}}.
\end{align*}
Using the fact $\sqrt{1-a} \leq 1-a/2$ for $a \in [0,1]$, we have the expression above is at least
\begin{align*}
\frac{\sqrt{k^2-k_1 k_2}}{2(k^2+k_1 k_2)}   (\alpha-\delta) k^2 \sqrt{\frac{2\log n}{k}} & \geq \frac{\sqrt{k^2-k (k-1)}}{4 k^2}   (\alpha-\delta) k^2 \sqrt{\frac{2\log n}{k}} \\
                                 &= \frac{\alpha-\delta}{4}\sqrt{2\log n}.
\end{align*}
For convenience of notation, denote $u(x)$ by
$$
u(x) = \frac{(\alpha-\delta)k^2 \sqrt{\frac{2\log n}{k}}-x}{\sqrt{k^2-k_1 k_2}}.
$$
Then we can further divide $I_1$ into two parts
\begin{align*}
\frac{1}{k \log n} \log I_1  =  o(1) + \frac{1}{k \log n} \log (I_{11} + I_{12})
\end{align*}
where
\begin{align*}
& I_{11} =  \int_{-k^2 (\log n)^{2/3}}^{\tau k^2 \sqrt{\frac{2\log n}{k}} } \frac{1}{2\pi u(x)^2} \frac{1}{\sqrt{2\pi k_1 k_2}}\exp \left( -\frac{((\alpha-\delta) k^2 \sqrt{2\log n/k}  - x)^2}{2(k^2-k_1k_2)} \times 2 - \frac{x^2}{2k_1 k_2}  \right)   dx, \\
& I_{12} = \int_{-\infty}^{-k^2 (\log n)^{2/3}} \frac{1}{2\pi u(x)^2} \frac{1}{\sqrt{2\pi k_1 k_2}}\exp \left( -\frac{((\alpha-\delta) k^2 \sqrt{2\log n/k}  - x)^2}{2(k^2-k_1k_2)} \times 2 - \frac{x^2}{2k_1 k_2}  \right)   dx.
\end{align*}
Since for any $x \in [-k^2 (\log n)^{2/3}, \tau k^2 \sqrt{\frac{2\log n}{k}}]$
$$
\frac{1}{k\log n}\log(u(x)^2) = o(1),
$$
we have
\begin{align}
\label{eq: logI11}
& \frac{1}{k \log n} \log I_{11}   \nonumber \\
&= o(1) + \frac{1}{k \log n} \log \int_{-k^2 (\log n)^{2/3}}^{\tau k^2 \sqrt{\frac{2\log n}{k}} } \frac{1}{\sqrt{2\pi k_1 k_2}}\exp \left( -\frac{((\alpha-\delta) k^2 \sqrt{2\log n/k}  - x)^2}{2(k^2-k_1k_2)} \times 2 - \frac{x^2}{2k_1 k_2}  \right)   dx  \nonumber \\
&=  o(1) -2(\alpha-\delta)^2 \frac{k^2}{k^2 + k_1 k_2} \nonumber \\
														& \quad   + \frac{1}{k \log n} \log  \int_{-k^2 (\log n)^{2/3}}^{\tau  k^2 \sqrt{\frac{2\log n}{k}}} \frac{1}{\sqrt{2 \pi \frac{ k_1 k_2(k^2-k_1k_2)}{k^2 + k_1 k_2}}}\exp \left( -\frac{\left(x - \frac{2k_1 k_2 k^2 (\alpha-\delta) \sqrt{2 \log n /k}}{k^2+k_1 k_2}  \right)^2}{2 \frac{ k_1 k_2(k^2-k_1k_2)}{k^2 + k_1 k_2}}  \right)   dx.
\end{align}
It follows from $\tau = (\alpha-\delta) \sqrt{2k_1 k_2 /(k^2 + k_1 k_2)}$ and $\sqrt{a}>a$ for $a\in(0,1)$ that
\begin{align}
\label{eq:IntegralRightBoundary}
\tau k^2 \sqrt{\frac{2\log n}{k}}  \geq  \frac{2k_1 k_2 k^2 (\alpha-\delta) \sqrt{2 \log n /k}}{k^2+k_1 k_2}.
\end{align}
Also we have as $n \rightarrow \infty$
\begin{align}
\label{eq:IntegralLeftBoundary}
\frac{-k^2 (\log n)^{2/3} - \frac{2k_1 k_2 k^2 (\alpha-\delta) \sqrt{2 \log n /k}}{k^2+k_1 k_2}}{\sqrt{\frac{ k_1 k_2(k^2-k_1k_2)}{k^2 + k_1 k_2}}} \rightarrow -\infty.
\end{align}
Observe that the integrand in (\ref{eq: logI11}) is a density function of a normal random variable. Then (\ref{eq:IntegralRightBoundary}) and (\ref{eq:IntegralLeftBoundary}) implies that the integral in (\ref{eq: logI11}) is in $[1/2+o(1),1]$. The last term in (\ref{eq: logI11}) is $o(1)$ and thus
$$
 \frac{1}{k \log n} \log I_{11} = o(1) -   2(\alpha-\delta)^2 \frac{k^2}{k^2 + k_1 k_2}.
$$
Also we have
$$
 \frac{1}{k \log n} \log I_{12} \leq \frac{1}{k \log n} \log \int_{-\infty}^{-k^2 (\log n)^{2/3}} \exp \left( - \frac{x^2}{2k_1 k_2}  \right)   dx.
$$
where the right hand size goes to $-\infty$ as $n \rightarrow \infty$. Using the approximation in (\ref{approx_Phi_u}) again and $\tau = (\alpha-\delta) \sqrt{2k_1 k_2 /(k^2 + k_1 k_2)}$, we have
\begin{align*}
\frac{1}{k \log n} \log I_2  & \leq \frac{1}{k \log n} \log  \int_{\tau k^2 \sqrt{\frac{2\log n}{k}} }^{\infty} \exp \left(-\frac{x^2}{2k_1k_2} \right) dx  \\
                             & =  o(1) - \tau^2 \frac{k^2}{k_1 k_2}         \\
														 &= o(1) - 2(\alpha-\delta)^2 \frac{k^2}{ k^2 + k_1 k_2}.
\end{align*}
Using $  \log (\max (a, b)) \leq \log (a + b ) \leq \log (2 \max (a,b))$ for $a,b>0$, we conclude
\begin{align}
\label{eq: ProbApprox}
& \frac{1}{k \log n} \log \mathbb{P} \left(X \in \left[(\alpha-\delta)k^2
																													\sqrt{\frac{2\log n}{k}}, (\alpha+\delta)k^2
																													\sqrt{\frac{2\log n}{k}} \right] \right)   \nonumber \\
& =   \frac{1}{k \log n} \log (I_1 + I_2)  \nonumber \\
& =  o(1) + \frac{1}{k \log n} \max (\log I_1, \log I_2) \nonumber \\
& = o(1) + \frac{1}{k \log n} \max (\log I_{11}, \log I_{12}, \log I_2) \nonumber \\
& = o(1) - 2(\alpha-\delta)^2 \frac{k^2}{ k^2 + k_1 k_2}.
\end{align}
For the special case $k_1 = k_2 =k$, the equation above still holds as shown earlier.

Now we estimate the first two terms in (\ref{exp_pair}). Let $\beta_1 \triangleq k_1/k$ and $\beta_2 \triangleq k_2/k$. Using the Stirling's approximation $a! \approx \sqrt{2\pi a} (a/e)^a$, $(n-b)\log(n-b) = (n-b) \log n - b(1+o(1))$ for $b=O(\log n)$ and $k \leq c\log n$, taking $\log$ of the first two terms in the right hand side of (\ref{exp_pair}) gives
\begin{align}
\label{ub_twoterms}
& \log \left(\frac{n!}{(k-k_1)! k_1! (k-k_1)!(n-2k+k_1)!}\frac{n!}{(k-k_2)! k_2! (k-k_2)! (n-2k+k_2)!}  \right)  \nonumber \\
=& O(1) + \log \bigg( \frac{ \sqrt{2 \pi n} n^{n}            }{ 2\pi (k-k_1) (k-k_1)^{2(k-k_1)} \sqrt{2\pi k_1} k_1^{k_1} \sqrt{2\pi(n-2k+k_1)} (n-2k+k_1)^{n-2k+k_1}                  } \times  \nonumber \\
& \quad \quad  \qquad    \qquad             \times \frac{ \sqrt{2 \pi n} n^{n}            }{ 2\pi (k-k_2) (k-k_2)^{2(k-k_2)} \sqrt{2\pi k_2} k_2^{k_2} \sqrt{2\pi(n-2k+k_2)} (n-2k+k_2)^{n-2k+k_2}}
\bigg)  \nonumber\\
=& O(1) + ( \log n + 2 n \log n ) - (\log (k-k_1) + 2(k-k_1)\log(k-k_1) ) - ( \frac{1}{2} \log k_1 + k_1\log k_1 ) \nonumber \\
   & - ( \frac{1}{2} \log (n-2k+k_1) + (n-2k+k_1) \log (n-2k+k_1) )  - (\log (k-k_2) + 2(k-k_2)\log(k-k_2) ) \nonumber \\
	 & - ( \frac{1}{2} \log k_2 + k_2\log k_2 ) - ( \frac{1}{2} \log (n-2k+k_2) + (n-2k+k_2) \log (n-2k+k_2) )                                                                      \nonumber \\
= & o(1) k \log n + ( \log n + 2 n \log n ) - ( \frac{1}{2} \log (n-2k+k_1) + (n-2k+k_1) \log (n-2k+k_1) ) \nonumber \\
&\quad - ( \frac{1}{2} \log (n-2k+k_2) + (n-2k+k_2) \log (n-2k+k_2) )          \nonumber \\
=&  (4-\beta_1-\beta_2+o(1)) k \log n.
\end{align}
Then it follows from (\ref{ub_twoterms}) and (\ref{eq: ProbApprox}) that
\begin{align*}
\frac{1}{k \log n} \log \mathbb{E}(\lvert \mathcal{O}(n,k,\alpha,k_1,k_2)\rvert)  & =  \sup_{\substack{\beta_1 \in (y_1-\delta, y_1+\delta)  \\
              \beta_2 \in (y_2-\delta, y_2+\delta) }} 4 - \beta_1 - \beta_2 -\frac{2}{1+\beta_1 \beta_2} (\alpha-\delta)^2 + o(1)  		\\
														           & = \sup_{\substack{\beta_1 \in (y_1-\delta, y_1+\delta)  \\
              \beta_2 \in (y_2-\delta, y_2+\delta) }} f(\alpha-\delta,\beta_1,\beta_2) + o(1)
\end{align*}
where the region of $(\beta_1, \beta_2)$ for the $\sup$ above comes from range of the sum in (\ref{exp_pair}). Then (\ref{eq:ConvergeExpectation}) follows from the continuity of $f(\alpha, y_1, y_2)$. This completes the proof of Theorem \ref{theorem:Number of pairs}.
\end{proof}

\section{Analysis of the $\LAS$ algorithm}\label{section:LAS}
\subsection{Preliminary results}\label{section:Preliminary}
We denote by $I_r^n$ the set of rows produced by the $\LAS$ algorithm in iterations $2r, r=0,1,\ldots$ and by $J_r^n$ the set of columns produced by $\LAS$
in iterations $2r-1, r=1,2,\ldots$. Without the loss of generality we set $I_0=J_0=\{1,\ldots,k\}$. Then $J_1$ is obtained by searching the $k$ columns with largest
sum of entries in the submatrix $\Cmatrix^{k\times n}$. Furthermore, $\Cmatrix_{2r+1}^n=\Cmatrix_{I_r^n,J_{r+1}^n}^n, r\ge 0$, and
$\Cmatrix_{2r}^n=\Cmatrix_{I_r^n,J_r^n}^n, r\ge 1$.

Next, for every $r$, denote by $\tilde J_r^n$ the set of $r$ columns with largest sum of entries in the $k\times (n-k)$ matrix
$\Cmatrix_{I_r^n,[n]\setminus J_r^n}$. In particular, in iteration $2r+1$ the algorithm chooses the best $k$ columns $J_{r+1}^n$
($k$ columns with largest entry sums)
from the $2k$ columns, the $k$ of which are the columns of $\Cmatrix_{I_r^n,J_r^n}$, and the remaining $k$ of which are columns
of $\Cmatrix_{I_r^n,[n]\setminus J_r^n}$.
Similarly, we define $\tilde I_r^n$ to be the set of $k$ rows with largest sum of entries in the $(n-k) \times k$ matrix
$\Cmatrix_{[n]\setminus I_r^n,J_{r+1}^n}$.

The following definition was introduced in \cite{bhamidi2012energy}:
\begin{definition}
\label{def}
Let $I$ be a set of $k$ rows and $J$ be a set of $k$ columns in $\Cmatrix^n$.
The submatrix $[\Cmatrix^n_{ij}]_{i\in I, j \in J}$ is defined to be row dominant in $\Cmatrix^n$ if
$$
\min_{i\in I}  \bigg\{\sum_{j\in J}\Cmatrix_{ij}^n  \bigg\} \geq \max_{i\in [n]\backslash I} \bigg\{\sum_{j\in J}\Cmatrix_{ij}^n  \bigg\}
$$
and is column dominant in $\Cmatrix^n$ if
$$
\min_{j\in J} \bigg\{\sum_{i\in I}\Cmatrix_{ij}^n \bigg\} \geq \max_{j\in [n]\backslash J}  \bigg\{\sum_{i\in I}\Cmatrix_{ij}^n \bigg\}.
$$
A submatrix which is both row dominant and column dominant is called a locally maximum submatrix.
\end{definition}
From the definition above, the $k \times k$ submatrix $\LAS$ returns in each iteration is either row dominant or column dominant, and the final submatrix the $\LAS$
converges to is a locally maximum submatrix.

We now recall the Analysis of Variance (ANOVA) Decomposition of a matrix. Given any $k \times k$ matrix $B$, let $B_{i\cdot}$ be the average of the $i$th row , $B_{\cdot j}$ be the average of the $j$th column, and $B_{\cdot \cdot}:=\avg(B)$ be the average of the matrix $B$. Then   the ANOVA decomposition $\Anv(B)$ of the matrix $B$ is defined as
\begin{align}
\label{eq:ANOVA}
\Anv(B)_{ij}=B_{ij}-B_{i\cdot}-B_{\cdot j} + B_{\cdot \cdot}, \; 1\leq i, j \leq k.
\end{align}
The matrix $B$ can then be rewritten as
\begin{align}
\label{eq:ANOVA_decomposition}
B=\avg(B)\mathbf{1}\mathbf{1}' + \text{\rm Row}(B) + \text{\rm Col}(B) + \Anv(B)
\end{align}
where $\text{\rm Row}(B)$ denotes the matrix with the $i$th row entries all equal to $B_{i\cdot}-B_{\cdot \cdot}$ for all $1\leq i \leq k$, and similarly $\text{\rm Col}(B)$ denotes the matrix with the $j$th column entries all equal to $B_{\cdot j}-B_{\cdot \cdot}$ for all $1\leq j \leq k$. An essential property of ANOVA decomposition is that, if $B$ consists of independent standard Gaussian variables, the random variables and matrices $B_{\cdot \cdot}, \text{\rm Row}(B),\text{\rm Col}(B)$ and $\Anv(B)$ are independent. This property is easily verified by establishing that the corresponding covariances are zero.

Recall the definition of $b_n$ in (\ref{eq:bn}). Let $L_n$ be the maximum of $n$ independent standard normal random variables. It is known that \cite{Leadbetter1983}
\begin{align}\label{eq:ExtremeGaussian1dim}
\sqrt{2\log n}(L_n-b_n)\Rightarrow  -\log G,
\end{align}
as $n\rightarrow\infty$, where $G$ is an exponential random variable with parameter $1$.

Let $(S_1,S_2)$ be a pair of positive random variables with joint density
\begin{align}
\label{eq:fgt}
f(s_1,s_2)  = C (\log(1+s_2/s_1))^{k-1}s_1^{k-1}e^{-(s_1+s_2)},
\end{align}
where $C$ is the normalizing constant to make $f(s_1,s_2)$ a density function. Let $\U=(U_1,\ldots,U_k)$ be a random vector with the Dirichlet distribution with parameter $1$.
Namely $\U$ is uniformly distributed on the simplex $\{(x_1,\cdots,x_k) \mid \sum_{i=1}^k x_i=1, x_i\ge 0, 1\le i\le k\}$.
Let
\begin{align*}
\Cmatrix_\infty^{\rm Row}\triangleq \left( -\log G, \log(1+S_1/S_2)\left(k\U-1\right) \mathbf{1}^T,  \text{\rm Col}(\Cmatrix^k), \Anv(\Cmatrix^k)\right),
\end{align*}
and
\begin{align*}
\Cmatrix_\infty^{\rm Col}\triangleq \left( -\log G, \text{\rm Row}(\Cmatrix^k), \log(1+S_1/S_2) \mathbf{1} \left(k\U-1\right)^T,  \Anv(\Cmatrix^k)\right),
\end{align*}
where $G,(S_1,S_2),\U$ are independent and distributed as above, and as before $\Cmatrix^k$ is a $k\times k$ matrix of i.i.d. standard normal random variables
independent from $G,(S_1,S_2),\U$.

Denote by $\mathcal{RD}_n$ the event that the matrix $\Cmatrix^k$ (the top $k\times k$
matrix of $\Cmatrix^n$) is row dominant. Similarly denote by $\mathcal{CD}_n$ the event that the same matrix is column dominant.
Let $\Dmatrix_{\rm row}^n$ be a random $k\times k$ matrix distributed as $\Cmatrix^k$ conditioned on the event $\mathcal{RD}_n$. Similarly define  $\Dmatrix_{\rm col}^n$.

Introduce the following two operators acting on $k\times k$ matrices $A$:
\begin{align}\label{eq:PsiOfA}
\Psi_n^{\text{\rm Row}}(A)&\triangleq
\left(\sqrt{2\log n}(\sqrt{k} \; \text{\rm ave}(A)-b_n), \sqrt{2k\log n}\text{\rm Row}(A), \text{\rm Col}(A), \Anv(A)\right)\in \R\times (\R^{k\times k})^3, \\
\Psi_n^{\text{\rm Col}}(A)&\triangleq
\left(\sqrt{2\log n}(\sqrt{k} \; \text{\rm ave}(A)-b_n), \text{\rm Row}(A), \sqrt{2k\log n}\text{\rm Col}(A), \Anv(A)\right)\in \R\times (\R^{k\times k})^3.
\end{align}
As a result, writing $\Psi_n^{\text{\rm Row}}(A)=(\Psi_{n,j}^{\text{\rm Row}}(A),~1\le j\le 4)$ and
applying (\ref{eq:ANOVA_decomposition}), we have
\begin{align}
A=\left({\Psi_{n,1}^{\text{\rm Row}}(A) \over \sqrt{2k\log n}}+{b_n\over \sqrt{k}}\right)\mathbf{1}\mathbf{1}'+{\Psi_{n,2}^{\text{\rm Row}}(A) \over \sqrt{2k\log n}}
+\Psi_{n,3}^{\text{\rm Row}}(A) +\Psi_{n,4}^{\text{\rm Row}}(A). \label{eq:AinTermsOfPsi}
\end{align}
A similar expression holds for $A$ in terms of $\Psi_n^{\text{\rm Col}}(A)$.

Bhamidi, Dey and Nobel (\cite{bhamidi2012energy}) established the limiting distribution result for locally maximum submatrix. For row (column) dominant submatrix, the following result can be easily derived following similar proof.
\begin{theorem}\label{theorem:RD-CD}
For every  $k>0$, the following convergence in distribution takes place as $n\rightarrow\infty$:
\begin{align}
\Psi_n^{\text{\rm Row}}(\Dmatrix_{\rm row}^n)
\Rightarrow
 \Cmatrix_\infty^{\rm Row}.       \label{eq:RD}
\end{align}
Similarly,
\begin{align}
\Psi_n^{\text{\rm Col}}(\Dmatrix_{\rm col}^n)
\Rightarrow
 \Cmatrix_\infty^{\rm Col}. \label{eq:CD}
\end{align}
\end{theorem}
Applying ANOVA decomposition (\ref{eq:ANOVA_decomposition}), the
result can be interpreted loosely as follows. $\Dmatrix_{\rm row}^n$ is approximately
\begin{align*}
\Dmatrix_{\rm row}^n \approx \sqrt{2\log n\over k} \mathbf{1}\mathbf{1}' +\text{\rm Col}(\Cmatrix^k)+\Anv(\Cmatrix^k)+O\left({\log\log n\over \sqrt{\log n}}\right).
\end{align*}
Indeed the first component of convergence (\ref{eq:RD}) means
\begin{align*}
\text{\rm avg}(\Dmatrix_{\rm row}^n)\approx {b_n \over \sqrt{k}}-{\log G \over \sqrt{2k\log n}}=\sqrt{2\log n\over k}+O\left({\log\log n\over \sqrt{\log n}}\right),
\end{align*}
and the second component of the same convergence means
\begin{align*}
\text{\rm Row}(\Dmatrix_{\rm row}^n)=O\left({1\over \sqrt{\log n}}\right).
\end{align*}

\subsection{Conditional distribution of the row-dominant and column-dominant submatrices}
Our next goal is to establish a conditional version of the Theorem~\ref{theorem:RD-CD}. We begin with several preliminary steps.
\begin{lemma}\label{lemma:CorrelatedColumns}
Fix a sequence $Z_1,\ldots,Z_n$ of i.i.d. standard normal random variables and $r$ distinct subsets $I_1,\ldots,I_r\subset [n], |I_\ell|=k, 1\le \ell \le r$.
Let $Y_\ell=k^{-{1\over 2}}\sum_{i\in I_\ell}Z_i$. Then there exists a lower triangular matrix
\begin{align}
\label{LL}
L=\left(\begin{array}{ccccc}
L_{1,1} & 0 & 0 &  \cdots & 0  \\
L_{2,1} & L_{2,2}  & 0&\cdots & 0  \\
%L_{q-4,q} & L_{q-4,q-2} & L_{q-4,q-4} & 0 & \cdots & 0 \\
\vdots & \vdots &  \vdots &  \ddots   & \vdots \\
L_{r,1} & L_{r,2} & L_{r,3}& \cdots & L_{r,r}
\end{array} \right),
\end{align}
such that
\begin{enumerate}
\item[(a)] $\left(Y_1,\ldots,Y_r\right)^T$ equals in distribution to $L (Y_1,W_2,\ldots,W_r)^T$, where $W_2,\ldots,W_r$ are i.i.d. standard normal
random variables independent from $Y_1$.

\item[(b)] The values $L_{i,j}$ are determined by the cardinalities of the intersections $I_{\ell_1}\cap I_{\ell_2}, 1\le \ell_1,\ell_2\le k$.

\item[(c)]
$L_{i,1}\in \{0,1/k,\ldots,(k-1)/k,1\}$ for all $i$, with
$L_{1,1}=1$, and $L_{i,1}\le (k-1)/k$, for all $i=2,\ldots,r$,

\item[(d)]
$\sum_{1\le i\le r}L_{\ell,i}^2=1$ for each $\ell=1,\ldots,r$.
\end{enumerate}

\end{lemma}

Note that $Y_1,\ldots,Y_r$ are correlated standard normal random variables. The lemma effectively provides a representation of these variables as a linear
operator acting on independent standard normal random variables, where since by condition (d) we have $L_{1,1}=1$, the first component $Y_1$ is preserved.

\begin{proof}
Let $\Sigma$ be the covariance matrix of $(Y_1,\ldots,Y_r)$ and let $\Sigma=LL^T$ be its Cholesky factorization. We claim that $L$ has the required property.
Note that the elements of $\Sigma$ are completely determined by the cardinalities of intersections $I_\ell \cap I_{\ell'}, 1\le \ell,\ell'\le r$
and thus (b) holds.
Since $\Sigma$ is the covariance matrix of $(Y_1,\ldots,Y_r)$ we obtain that this vector equals in distribution $L(W_1,\ldots,W_r)^T$, where $W_i, 1\le i\le r$
are i.i.d. standard normal and thus (a) holds.
 We can take $W_1$ to be $Y_1$ since $Y_1$ is also a standard normal. Note that $L_{1,1}$ is the variance of $Y_1$ hence $L_{1,1}=1$.
The variance of $Y_\ell$ is $\sum_{1\le i\le r}L_{\ell,i}^2$ which equals $1$ since $Y_\ell$ is also standard normal, namely (d) holds. Finally, note that
$L_{i,1}$ is the covariance of $Y_1$ with $Y_i, i=2,\ldots,r$, which takes one of the values  $0,1/k,\ldots,(k-1)/k$, since $I_\ell$ are
distinct subset of $[n]$ with cardinality $k$. This establishes (c).
\end{proof}

Recall that  $\omega_n$ denotes any strictly increasing positive function satisfying
$\omega_n=o(\sqrt{2\log n})$ and $\log \log n = O(\omega_n)$. We now establish the following conditional version of (\ref{eq:ExtremeGaussian1dim}).
\begin{lemma}\label{lemma:RowMaxConditional}
Fix a positive integer $r \geq 2$ and $r\times r$ lower triangular matrix $L$ satisfying $|L_{\ell,i}| \leq 1$ and $L_{\ell,1} \leq (k-1)/k, \ell=2,\ldots,r$. Let $\Zmatrix=(Z_{i,\ell}, 1\le i\le n, 1 \le \ell\le r)$ be a matrix of i.i.d. standard normal random variables.
Given any $\bar c=(c_\ell, 1\le \ell\le r-1)\in\R^{r-1}$,
for each $i=1,\ldots,n$, let $\mathcal{B}_i=\mathcal{B}_i(\bar c)$ denote the event
\begin{align*}
\left[L\left(Z_{i,1},Z_{i,2},\ldots,Z_{i,r}\right)^T\right]_{\ell}
\le \sqrt{2\log n}+c_{\ell-1}, \qquad \forall~2\le \ell\le r,
\end{align*}
where $[\cdot]_\ell$ denotes the $\ell$-th component of  the vector in the argument. Then for every $w\in\R$
\begin{align*}
\lim_{n\rightarrow\infty}
\sup_{\bar c:\|\bar c\|_\infty\le \omega_n}
\Big|\pr\left(  \sqrt{2 \log n} \left(\max_{1\le i\le n} Z_{i,1} - b_n  \right)  \le w   \; \Big|\bigcap_{1\le i\le n} \mathcal{B}_i\right)
-\exp\left(-\exp\left(-w\right)\right)\Big|=0.
\end{align*}
\end{lemma}

Namely, the events $\mathcal{B}_i$ have an asymptotically negligible effect on the weak convergence fact (\ref{eq:ExtremeGaussian1dim}), namely that
$$
\sqrt{2 \log n} (\max_{1\le i\le n} Z_{i,1}-b_n)\Rightarrow -\log G.
$$
\begin{proof}
Note that the events $\mathcal{B}_i$, $1\leq i \leq n$ are independent. Thus we rewrite
\begin{align}
\label{conditional_prob}
& \pr\left(  \sqrt{2 \log n} \left(\max_{1\le i\le n} Z_{i,1} - b_n  \right)  \le w  \;  \Big|\bigcap_{1\le i\le n} \mathcal{B}_i\right)  = \pr \left( \max_{1\le i\le n} Z_{i,1} \leq b_n + \frac{w}{\sqrt{2 \log n}} \; \Big|\bigcap_{1\le i\le n} \mathcal{B}_i           \right)  \nonumber\\
& = \pr \left(Z_{1,1} \leq b_n + w/\sqrt{2 \log n} \; \mid \; \mathcal{B}_i \right)^n  \nonumber\\
& = \left( 1 - \frac{\pr \left(Z_{1,1} > b_n + w/\sqrt{2 \log n}, \mathcal{B}_1 \right) }{\pr(\mathcal{B}_1)}  \right)^n
\end{align}
Fix any $\delta_1, \delta_2 \in (0,1/(2k))$. Let $\tilde{\mathcal{B}}_{1} = \tilde{\mathcal{B}}_{1}(\delta_1, \delta_2)$ be the event that
\begin{align*}
Z_{1,1} \leq (1+\delta_2) b_n \text{ and } |Z_{1,l}| \leq \frac{\delta_1}{r-1} b_n, \qquad \forall~2\le \ell\le r.
\end{align*}
We claim that $\tilde{\mathcal{B}}_1 \subset \mathcal{B}_1$ for all large enough $n$ and any $\bar{c}$ satisfying $\|\bar{c}\|_{\infty} \leq \omega_n$. Indeed, using $L_{\ell,1} \leq (k-1)/k$ and $|L_{\ell,i}| \leq 1, \ell=2,\ldots,r$, the event $\tilde{\mathcal{B}}_{1}$ implies
\begin{align*}
L_{\ell,1} Z_{1,1} + \sum_{i=2}^\ell L_{\ell,i} Z_{1,\ell} & \leq (1-1/k)(1+\delta_2)b_n + \delta_1 b_n, \qquad \forall~2\le \ell \le r.
\end{align*}
Then for any $\bar{c}$ satisfying $\|\bar{c}\|_{\infty} \leq \omega_n$, we can choose sufficiently large $n$ such that
$$
(1-1/k)(1+\delta_2)b_n + \delta_1 b_n  \leq \sqrt{2 \log n} + c_{\ell-1}, \qquad \forall~2\le \ell \le r,
$$
from which the claim follows. Then we have
\begin{align}
\label{sandwich_inequality}
1-\pr(Z_{1,1}>b_n+w/\sqrt{2\log n}, \tilde{\mathcal{B}}_1) \geq 1 - \frac{\pr(Z_{1,1} > b_n + w/\sqrt{2 \log n}, \mathcal{B}_1)}{\pr(\mathcal{B}_1)} \geq 1 - \frac{\pr(Z_{1,1} > b_n+w/\sqrt{2\log n})}{\pr(\tilde{\mathcal{B}}_1)}.
\end{align}
Using (\ref{approx_Phi_u}), we simplify
\begin{align*}
\pr(Z_{1,1}>b_n + w/\sqrt{2 \log n}, \tilde{\mathcal{B}}_1) & = \pr( (1+\delta_2)b_n\geq Z_{1,1}>b_n + w/\sqrt{2 \log n}) \pr \left(|Z_{1,\ell}| \leq \frac{\delta_1}{r-1} b_n \right)^{r-1}   \\
& = \frac{1}{(b_n + w/\sqrt{2 \log n})\sqrt{2\pi}}\exp\left( -\frac{(b_n + w/\sqrt{2 \log n})^2}{2} \right)(1+o(1)).
\end{align*}
Also using $\lim_{n \rightarrow \infty}\pr(\tilde{\mathcal{B}_1}) = 1$, we simplify
\begin{align*}
\frac{\pr(Z_{1,1}>b_n + w/\sqrt{2 \log n})}{\mathbb{P}(\tilde{\mathcal{B}}_1)} = \frac{1}{(b_n + w/\sqrt{2 \log n})\sqrt{2\pi}}\exp\left( -\frac{(b_n + w/\sqrt{2 \log n})^2}{2} \right)(1+o(1)).
\end{align*}
The two equations above give the same asymptotics of the two sides in (\ref{sandwich_inequality}). Hence the term in the middle also has the same asymptotics
\begin{align}
\label{approx_B1}
1 - \frac{\pr(Z_{1,1} > b_n + w/\sqrt{2 \log n}, \mathcal{B}_1)}{\pr(\mathcal{B}_1)} &= 1-\frac{1}{(b_n + w/\sqrt{2 \log n})\sqrt{2\pi}}\exp\left( -\frac{(b_n + w/\sqrt{2 \log n})^2}{2} \right)(1+o(1))  \nonumber \\
& = 1-\mathbb{P}(Z_{1,1}>b_n + w/\sqrt{2 \log n})(1+o(1))
\end{align}
Substituting (\ref{approx_B1}) into (\ref{conditional_prob}), we have for any $\bar{c}$ satisfying $\|\bar{c}\|_{\infty} \leq \omega_n$
\begin{align*}
& \lim_{n \rightarrow \infty} \pr\left(  \sqrt{2 \log n} \left(\max_{1\le i\le n} Z_{i,1} - b_n  \right)  \le w \Big| \; \bigcap_{1\le i\le n} \mathcal{B}_i \right)  = \lim_{n \rightarrow \infty} (1-\mathbb{P}(Z_{1,1}>b_n + w/\sqrt{2 \log n}) )^n \\
& =\lim_{n \rightarrow \infty} \pr \left( \sqrt{2 \log n} \left(\max_{1\le i\le n} Z_{i,1} - b_n  \right)  \le w  \right)
\end{align*}
By the limiting distribution of the maximum of $n$ independent standard Gaussians, namely (\ref{eq:ExtremeGaussian1dim}),
$$
\lim_{n \rightarrow \infty} \pr \left( \sqrt{2 \log n} \left(\max_{1\le i\le n} Z_{i,1} - b_n  \right)  \le w  \right) = \exp(-\exp(-w)).
$$
Then the result follows.
\end{proof}

We now state and prove the main result of this section - the conditional version of Theorem~\ref{theorem:RD-CD}. By Portmanteau's theorem, a weak convergence
$X_n\Rightarrow X$ is established by showing $\E[f(X_n)]\rightarrow \E[f(X)]$ for every bounded continuous function $f$. We use this version
in the theorem below.

\begin{theorem}\label{theorem:RD-CD_conditional}
Fix a positive integer $r$ and for each $n$ fix any
distinct subsets $I_0,\ldots,I_{r-1}\subset [n], |I_\ell|=k, 0\le \ell\le r-1$,
and distinct subsets $J_1,\ldots,J_r\subset [n], |J_\ell|=k, 1\le \ell\le r$. Fix any sequence $C_1,\ldots,C_{2r-1}$ of $k\times k$ matrices
satisfying $\|C_\ell\|_\infty\le \omega_n, 1\le \ell\le 2r-1$.
Let
$$
 \mathcal{E}_r=\mathcal{E}\left(I_i, 1 \leq i \leq r; J_j, 1 \leq j \leq r; C_\ell, 1 \le \ell \le 2r - 1 \right)
$$
be the event that $\Cmatrix^n_{I_{\ell-1},J_\ell}-\sqrt{2\log n\over k} \mathbf{1} \mathbf{1}' =C_{2\ell-1}$ for each $1\le \ell\le r$,
$\Cmatrix^n_{I_{\ell},J_\ell}-\sqrt{2\log n\over k} \mathbf{1} \mathbf{1}' =C_{2\ell}$ for each $1\le \ell\le r-1$,
and, furthermore,
$\sqrt{2\log n\over k} \mathbf{1} \mathbf{1}' +C_\ell$ is the $\ell$-th matrix returned by the algorithm $\LAS$ for all $\ell=1,\ldots,2r-1$. Namely, $\Cmatrix_\ell^n=\sqrt{2\log n\over k} \mathbf{1} \mathbf{1}' +C_\ell$.
Fix any set of columns $J\subset [n], |J|=k$ such that
$J\setminus \left(\cup_{1\le \ell\le r-1}J_\ell\right)\ne \emptyset$, including possibly $J_r$, and let $\Dmatrix_{\rm Row}^n$ be the $k\times k$ submatrix of
$\Cmatrix_{([n] \setminus I_{r-1}),J}^n$ with the largest average value and $\hat \Dmatrix^n_{\rm Row}$ be the $k\times k$ submatrix of
$\Cmatrix_{\left([n]\setminus \cup_{0\le\ell\le r-1}I_\ell\right),J}^n$ with the largest average value. Then, the following holds.

\begin{enumerate}
\item [(a)]
\begin{align}\label{eq:hatDisD}
\lim_{n\rightarrow\infty}\inf\pr\left(\hat \Dmatrix^n_{\rm Row}=\Dmatrix^n_{\rm Row}|\mathcal{E}_r\right)=1,
\end{align}
where $\inf$ is over all $I_\ell,J_\ell$ and $C_\ell, 1\le \ell\le 2r-1$ satisfying $\|C_\ell\|_\infty\le \omega_n$.

\item[(b)] Conditional on $\mathcal{E}_r$, $\Psi_n^{\text{\rm Row}}(\Dmatrix^n_{\rm Row})$ converges to $\Cmatrix_\infty^{\rm Row}$ uniformly in $(C_\ell, 1\leq l \leq 2r-1)$.
Specifically,
for every bounded continuous function  $f:\R\times\left(\R^{k\times k}\right)^3\rightarrow\R$
(and similarly to (\ref{eq:RD})) we have
\begin{align}
\lim_{n\rightarrow\infty}\sup \Big|\E\left[f\left(\Psi_n^{\text{\rm Row}}(\Dmatrix^n_{\rm Row})\right) | \mathcal{E}_r\right]
-\E\left[f\left(\Cmatrix_\infty^{\rm Row}\right)\right]\Big|=0, \label{eq:RD-conditional}
\end{align}
where $\sup$ is over all $I_\ell,J_\ell$ and $C_\ell, 1\le \ell\le 2r-1$ satisfying $\|C_\ell\|_\infty\le \omega_n$.

\item[(c)]
\begin{align*}
\lim_{n\rightarrow\infty}\inf\pr\left(\|\Dmatrix^n_{\rm Row}-\sqrt{2\log n\over k}\|_\infty\le \omega_n|\mathcal{E}_r\right)=1,
\end{align*}
where $\inf$ is over all $I_\ell,J_\ell$ and $C_\ell, 1\le \ell\le 2r-1$ satisfying $\|C_\ell\|_\infty\le \omega_n$.

\end{enumerate}
Similar results of (a), (b) and (c) hold for $\Dmatrix^n_{\rm Col}$, $\hat{\Dmatrix}^n_{\rm Col}$ and $\Psi_n^{\text{\rm Col}}(\Dmatrix^n_{\rm Col})$
when $I\subset [n], |I|=k$ is such that
$I\setminus \left(\cup_{0 \le \ell\le r-1}I_\ell\right)\ne \emptyset$,  $\Dmatrix^n_{\rm Col}$ is the $k\times k$ submatrix of
$\Cmatrix_{I,([n]\setminus J_r)}^n$ with the largest average value and $\hat{\Dmatrix}^n_{\rm Col}$ is the $k\times k$ submatrix of
$\Cmatrix_{I,([n]\setminus \cup_{1 \leq \ell \leq r}J_r)}^n$ with the largest average value.
\end{theorem}

Regarding the subset of columns $J$ in the theorem above, primarily the special case $J=J_r$ will be used. Note that indeed
$J_r\setminus \left(\cup_{1\le \ell\le r-1}J_\ell\right)\ne \emptyset$, by applying part (a) of the theorem to the previous step algorithm
which claims the identity $\hat\Dmatrix^n_{\rm Col}=\Dmatrix^n_{\rm Col}$ w.h.p.

\begin{proof}
Unlike for $\Dmatrix^n_{\rm Row}$, in the construction of $\hat \Dmatrix^n_{\rm Row}$ we only use rows $\Cmatrix_{i,J}^n$ which are outside the rows
$\cup_{0\le\ell\le r-1}I_\ell$ already used in the previous iterations of the algorithm. The bulk of the proof of the theorem will be to
establish that claims (b) and (c) of the theorem hold for this matrix instead. Assuming this is the case, (a) then implies that (b) and (c) hold for $\Dmatrix^n_{\rm Row}$ as well, completing the proof of theorem.

First we prove part (a) assuming (b) and (c) hold for $\hat \Dmatrix^n_{\rm Row}$. We fix any set of rows $I\subset [n] \setminus I_{r-1}$ with cardinality $k$ satisfying $I\cap \left(\cup_{0\le\ell\le r-2}I_\ell\right)\ne\emptyset$.
For every $i\in I\cap \left(\cup_{0\le\ell\le r-2}I_\ell\right)$ and $j\in J\cap (\cup_{1\le \ell\le r-1}J_{\ell}^n)$, $\Cmatrix_{i,j}^n$ is either included in some $\Cmatrix_{\ell}^n$, in which case $|\Cmatrix_{i,j}^n-\sqrt{2\log n / k}| \le \omega_n$ holds under the event $\mathcal{E}_r$, or $\Cmatrix_{i,j}^n$ is not included in any $\Cmatrix_{\ell}^n$, in which case $\Cmatrix_{i,j}^n$ is $O(1)$ w.h.p. under $\mathcal{E}_r$. Then in both cases we have
$$
{ \lim_{n\rightarrow\infty}
\inf\pr\left( \Cmatrix_{i,j}^n-\sqrt{2\log n\over k} \le \omega_n | ~\mathcal{E}_r\right) = 1,             }
$$
where $\inf$ is over all $I_\ell,J_\ell$ and $C_\ell, 1\le \ell\le 2r-1$ satisfying $\|C_\ell\|_\infty\le \omega_n$. Since $| \left(\cup_{0\le\ell\le r-2}I_\ell\right)| \leq (r-1)k$ and $r$ is fixed, by the union bound the same applies to all such elements $\Cmatrix_{i,j}^n$. By part (b) which was assumed to hold for $\hat \Dmatrix^n_{\rm Row}$, we have
\begin{align*}
\lim_{n\rightarrow\infty}
\inf\pr\left( \sum_{j\in J} \Cmatrix_{i,j}^n- k \sqrt{2\log n\over k} \le k\omega_n,
 ~~\forall i\in [n] \backslash (\cup_{0 \leq \ell \leq r-1} I_{\ell}) ~|~\mathcal{E}_r\right) = 1,
\end{align*}
where $\inf$ is over the same set of events as above. On the other hand for every $i\in I\cap \left(\cup_{0\le\ell\le r-2}I_\ell\right)$ and
$j\in J\setminus(\cup_{1\le \ell\le r-1}J_{\ell})$, $\Cmatrix_{i,j}^n$ is not included in any $\Cmatrix_\ell^n$, $1 \leq \ell \leq 2r -1$ and hence is $O(1)$ w.h.p. under the event $\mathcal{E}_r$, which gives
\begin{align}
\label{eq: CijLowbBound}
\lim_{n\rightarrow\infty}\sup\pr\left(\Cmatrix_{i,j}^n\ge (1/2)\sqrt{2\log n/k}~|~\mathcal{E}_r\right)=0.
\end{align}
Since $|\cup_{0\le\ell\le r-2}I_\ell|\le (r-1)k$
and $r$ is fixed, by the union bound the same applies to all such elements $\Cmatrix_{i,j}^n$. It follows, that w.h.p.
the average value of the matrix $\Cmatrix_{I,J}^n$ for all sets of rows $I \in [n] \setminus I_{r-1}$ satisfying $I \cap (\cup_{0\leq l \leq r-2} I_l) \neq \emptyset$ is at most $(1-1/(2k^2))\sqrt{2\log n/k}+\omega_n$, since by assumption $J \setminus (\cup_{1 \leq \ell \leq r-1} J_{\ell}) \neq \emptyset$ and thus
there exists at least one entry
in $\Cmatrix_{I,J}^n$ satisfying (\ref{eq: CijLowbBound}). On the other hand by part (b), the average value of $\hat\Dmatrix_{\rm Row}^n$ is at least $\sqrt{2\log n/k}-\omega_n$ and thus (\ref{eq:hatDisD}) in (a)  follows. The proof for $\hat\Dmatrix_{\rm Col}^n$ is similar.

Thus we now establish (b) and (c) for $\hat \Dmatrix^n_{\rm Row}$. In order to simplify the notation, we use $\Dmatrix^n_{\rm Row}$ in place of $\hat \Dmatrix^n_{\rm Row}$.
 We fix $I_\ell,J_\ell, C_\ell$ and $J$ as described in the assumption of the theorem.
Let $I^c=[n]\setminus (\bigcup_{0\le \ell \le r-1}I_\ell)$. For each
$i\in I^c$ consider the event denoted by $\mathcal{B}_i^{\rm Row}$ that for each  $\ell=1,\ldots,r-1$
$\Cmatrix_{I_\ell,J_\ell}^n=\sqrt{2\log n\over k} \mathbf{1} \mathbf{1}^T + C_{2\ell}$ and
\begin{align}\label{eq:Idominance}
\text{\rm Ave}\left(\Cmatrix_{i,J_\ell}^n\right)\le \min_{i'\in I_\ell}\text{\rm Ave}(\Cmatrix_{i',J_\ell}^n).
\end{align}
 Our key observation is that the  distribution of the submatrix
$\Cmatrix_{I^c,J}^n$ conditional on the event $\mathcal{E}_r$ is the same as the distribution of the same submatrix conditional on the event $\bigcap_{i\in I^c}\mathcal{B}_i^{\rm Row}$.
Thus we need to show the convergence in distribution of $\hat{\Dmatrix}_{\rm row}^n$ conditional on the event $\bigcap_{i\in I^c}\mathcal{B}_i^{\rm Row}$.
A similar observation holds for the column version of the statement which we skip.

Now fix any $i\in I^c$. Let $J_0=J$ for convenience, and consider the $r$-vector
\begin{align}\label{eq:vector Y}
\left(Y_\ell\triangleq k^{1\over 2}\text{\rm Ave}(\Cmatrix_{i,J_\ell}^n), 0\le\ell\le r-1\right).
\end{align}
Without any conditioning
the distribution of this vector is the distribution of standard normal random variables with correlation structure determined by the vector of
cardinalities of intersections of the sets $J_\ell$, namely vector
$\sigma\triangleq\left( |J_\ell\cap J_{\ell'}|, 0\le \ell,\ell'\le r-1\right)$. By Lemma~\ref{lemma:CorrelatedColumns} there exists a $r \times r$
matrix $L$ which depends on $\sigma$ only and with properties (a)-(d) described in the lemma,
such that the distribution of the vector (\ref{eq:vector Y}) is the same as the one of $LZ$, where $Z$ is the $r$-vector of i.i.d. standard normal random variables.
We will establish Theorem~\ref{theorem:RD-CD_conditional} from the following proposition, which is an analogue of Lemma \ref{lemma:RowMaxConditional}. We delay its proof for later.

\begin{prop}\label{prop:RD-CD_conditional}
Let $\Zmatrix=(Z_{i,\ell}, 1\le i\le n, 1\le \ell\le r-1)$ be a matrix of i.i.d. standard normal random variables independent from the $n\times k$ matrix $\Cmatrix^{n\times k}$.
Given any $\bar c=(c_\ell, 1\le \ell\le r-1)\in\R^{r-1}$,
for each $i=1,\ldots,n$, let $\mathcal{B}_i$ be the event
\begin{align*}
\left[L\left(k^{1\over 2}\Ave(\Cmatrix_{i,[k]}^{n\times k}),Z_{i,1},\ldots,Z_{i,r-1}\right)^T\right]_{\ell+1}
\le \sqrt{2\log n}+\sqrt{k}c_\ell, \qquad \forall~1\le \ell\le r-1,
\end{align*}
where $[\cdot]_\ell$ denotes the $\ell$-th component of  the vector in the argument.
For every bounded continuous function $f:\R\times \left(\R^{k\times k}\right)^3\rightarrow \R$
\begin{align}
\lim_{n\rightarrow\infty}\sup_{\bar c:\|\bar c\|_\infty\le \omega_n} \Big|\E\left[f\left(\Psi_n^{\rm Row}(\Cmatrix^k)\right)~|~\mathcal{RD}_n, \cap_{1\le i\le n}\mathcal{B}_i \right]
-\E\left[f(\Cmatrix_\infty^{\rm Row})\right]\Big|=0. \label{eq:EfConvergence}
\end{align}
\end{prop}
The proposition essentially says that the events $\mathcal{B}_i$ have an asymptotically negligible effect on the distribution of the largest $k\times k$ submatrix
of $\Cmatrix^{n\times k}$.

First we show how this proposition implies part (b) of Theorem~\ref{theorem:RD-CD_conditional}. The event $\bigcap_{i\in I^c} \mathcal{B}_{i}^{\rm Row}$
implies that $\|C_{2\ell}\|_\infty\le \omega_n$, for all $\ell$
and therefore
\begin{align*}
-\omega_n \le c_\ell\triangleq\min_{i'\in I_\ell}\text{\rm Ave}(\Cmatrix_{i',J_\ell}^n)-\sqrt{2\log n\over k} \le \omega_n, \qquad 1\le \ell\le r-1.
\end{align*}
The events $\bigcap_{i\in I^c} \mathcal{B}_{i}^{\rm Row}$ and $\bigcap_{1\le i\le n}\mathcal{B}_i$ are then identical modulo the difference
of cardinalities $|I^c|$ vs $n$. Since $k$ is a constant, then $|I^c|=n-O(1)$, and the result is claimed in the limit $n\rightarrow\infty$. The assertion (b) holds.

We now establish (c). Recalling the representation (\ref{eq:AinTermsOfPsi}) and the definition of $b_n$ we have
\begin{align*}
\Dmatrix_{\rm Row}^n-\sqrt{2\log n\over k} \mathbf{1} \mathbf{1}'={\Psi_{n,1}^{\text{\rm Row}}(\Dmatrix_{\rm Row}^n) \over \sqrt{2k\log n}} \mathbf{1} \mathbf{1}'
+{\Psi_{n,2}^{\text{\rm Row}}(\Dmatrix_{\rm Row}^n) \over \sqrt{2k\log n}}
+\Psi_{n,3}^{\text{\rm Row}}(\Dmatrix_{\rm Row}^n) +\Psi_{n,4}^{\text{\rm Row}}(\Dmatrix_{\rm Row}^n)
+O\left({\log\log n\over \sqrt{\log n}}\right).
\end{align*}
The claim then follows immediately from part (b), specifically from the uniform weak convergence
$\Psi_n^{\rm Row}(\Dmatrix_{\rm Row}^n)\Rightarrow \Cmatrix_\infty^{\rm Row}$.
\end{proof}

\begin{proof}[Proof of Proposition~\ref{prop:RD-CD_conditional}]
According to Theorem \ref{theorem:RD-CD}, for every bounded continuous function $f$,
\begin{align}
\lim_{n\rightarrow\infty} \E\left[f\left(\Psi_n^{\rm Row}(\Cmatrix^k)\right)~|~\mathcal{RD}_n \right]
= \E\left[f(\Cmatrix_\infty^{\rm Row})\right]. \label{eq:EfConvergence_step1}
\end{align}
Our goal is to show
\begin{align}
\lim_{n\rightarrow\infty}\sup_{\bar c:\|\bar c\|_\infty\le \omega_n} \Big|\E\left[f\left(\Psi_n^{\rm Row}(\Cmatrix^k)\right)~|~\mathcal{RD}_n, \cap_{1\le i\le n}\mathcal{B}_i \right]
-\E\left[f\left(\Psi_n^{\rm Row}(\Cmatrix^k)\right)~|~\mathcal{RD}_n \right]\Big|=0. \label{eq:EfConvergence_step2}
\end{align}
(\ref{eq:EfConvergence}) follows from (\ref{eq:EfConvergence_step1}) and (\ref{eq:EfConvergence_step2}). We claim that if the following relation holds for any $W \in \R\times (\R^{k\times k})^3$
\begin{align}
\lim_{n\rightarrow\infty}\sup_{\bar{c}:\|\bar{c}\|_{\infty}\le \omega_n}
\Big|{\pr\left(\bigcap_{1\le i\le n} \mathcal{B}_i|\Psi_n^{\rm Row}(\Cmatrix^k)=W, \mathcal{RD}_n\right)
\over\pr\left(\bigcap_{1\le i\le n} \mathcal{B}_i\right)}-1\Big|=0, \label{eq:BiConditioned}
\end{align}
then (\ref{eq:EfConvergence_step2}) follows. By symmetry
\begin{align*}
\pr\left(\mathcal{RD}_n|\bigcap_{1\le i\le n} \mathcal{B}_i\right)={n\choose k}^{-1}=\pr\left(\mathcal{RD}_n\right).
\end{align*}
Using the equation above, we compute
\begin{align}
\label{eq:EfCompute}
& \E\left[f\left(\Psi_n^{\rm Row}(\Cmatrix^k)\right)~|~\mathcal{RD}_n, \cap_{1\le i\le n}\mathcal{B}_i \right] = \int~f(W) {d\pr\left(\Psi_n^{\rm Row}(\Cmatrix^k)=W, \mathcal{RD}_n, \bigcap_{1\le i\le n} \mathcal{B}_i\right)
\over \pr\left(\mathcal{RD}_n,\bigcap_{1\le i\le n} \mathcal{B}_i\right)}  \nonumber\\
& = \int~f(W)  {\pr\left(\cap_{1\le i\le n}\mathcal{B}_i ~\big|~ \Psi_n^{\rm Row}(\Cmatrix^k)=W, \mathcal{RD}_n \right)
\over \pr\left(\bigcap_{1\le i\le n} \mathcal{B}_i\right)}  {d\pr\left(\Psi_n^{\rm Row}(\Cmatrix^k)=W, \mathcal{RD}_n\right)
\over \pr\left(\mathcal{RD}_n ~\big|~   \bigcap_{1\le i\le n} \mathcal{B}_i \right)}  \nonumber\\
& = \int~f(W)  {\pr\left(\cap_{1\le i\le n}\mathcal{B}_i ~\big|~ \Psi_n^{\rm Row}(\Cmatrix^k)=W, \mathcal{RD}_n \right)
\over \pr\left(\bigcap_{1\le i\le n} \mathcal{B}_i\right)}  d\pr\left(\Psi_n^{\rm Row}(\Cmatrix^k)=W ~\big|~ \mathcal{RD}_n\right)
\end{align}
Substituting (\ref{eq:EfCompute}) into the left hand side of (\ref{eq:EfConvergence_step2}) and then using (\ref{eq:BiConditioned}) and the boundedness of $f$, we obtain (\ref{eq:EfConvergence_step2}).

The rest of the proof is to show that (\ref{eq:BiConditioned}) holds for any $W \in \R\times (\R^{k\times k})^3$. Fix any  $W \triangleq (w_1,W_2,W_3,W_4)$ where $w_1 \in \R$ and $W_2$, $W_3$, $W_4 \in \R^{k \times k}$. Conditional on $\Psi_n^{\rm Row}(\Cmatrix^k)=W$, and writing $W_2=(W^2_{i,j})$ the average value of the $i$-th row of $\Cmatrix^k$ is
\begin{align*}
\Cmatrix_{i\cdot}^k=
{W^2_{i,1}\over \sqrt{2k\log n}}+{w_1\over \sqrt{2k\log n}}+{b_n\over\sqrt{k}}\triangleq w_{i,n}, \qquad i=1,\ldots,k.
\end{align*}
Let $c_n(W)=\min_{1\le i\le k}w_{i,n}$. Note that
\begin{align}
w_{i,n}=\sqrt{2\log n\over k}+o(1), ~c_n(W)=\sqrt{2\log n\over k}+o(1). \label{eq:scale of cn}
\end{align}
The event $\mathcal{RD}_n$
is equivalent to the event
\begin{align*}
\max_{k+1\le i\le n}\text{\rm Ave}(\Cmatrix_{i\cdot}^{n\times k})\le c_n(W).
\end{align*}

Now observe that by independence of rows of $\Zmatrix$
\begin{align}
&\pr\left(\bigcap_{1\le i\le n} \mathcal{B}_i|\Psi_n^{\rm Row}(\Cmatrix^k)=W, \max_{k+1\le i\le n}\text{\rm Ave}(\Cmatrix_{i\cdot}^{n\times k})\le c_n(W)\right) \notag\\
&=\pr\left(\bigcap_{1\le i\le k}\mathcal{B}_i|\Psi_n^{\rm Row}(\Cmatrix^k)=W\right)
\pr\left(\bigcap_{k+1\le i\le n} \mathcal{B}_i |\max_{k+1\le i\le n}\text{\rm Ave}(\Cmatrix_{i\cdot}^{n\times k})\le c_n(W)\right).
\label{eq:CapBi}
\end{align}
By (\ref{eq:ExtremeGaussian1dim}) we have
\begin{align*}
\lim_{n\rightarrow\infty}\pr\left(\max_{k+1\le i\le n}\text{\rm Ave}(\Cmatrix_{i\cdot}^{n\times k})\le c_n(W)\right)
&=
\lim_{n\rightarrow\infty}\pr\left(\max_{k+1\le i\le n}\sqrt{2\log n}\left(\sqrt{k}\text{\rm Ave}(\Cmatrix_{i\cdot}^{n\times k})-b_n\right)\le w_1+\min_{1\le i\le k}W_{i,1}^2\right)\\
&=\exp\left(-\exp\left(-w_1-\min_{1\le i\le k}W_{i,1}^2\right)\right),
\end{align*}
Furthermore, by Lemma \ref{lemma:RowMaxConditional} we also have
\begin{align*}
\lim_{n\rightarrow\infty}
\sup_{\bar{c}:\|\bar{c}\|_{\infty}\le \omega_n}
\Big|\pr\left(\max_{k+1\le i\le n}\text{\rm Ave}(\Cmatrix_{i\cdot}^{n\times k})\le c_n(W)|\bigcap_{k+1\le i\le n} \mathcal{B}_i\right)
-\exp\left(-\exp\left(-w_1-\min_{1\le i\le k}W_{i,1}^2\right)\right)\Big|=0.
\end{align*}
Applying Bayes rule, we obtain
\begin{align}
\label{eq:ClaimBi0}
\lim_{n\rightarrow\infty}\sup_{\bar{c}:||\bar{c}\|\le \omega_n}
\Big|{\pr\left(\bigcap_{k+1\le i\le n} \mathcal{B}_i |\max_{k+1\le i\le n}\text{\rm Ave}(\Cmatrix_{i\cdot}^{n\times k})\le c_n(W)\right)
\over \pr\left(\bigcap_{k+1\le i\le n} \mathcal{B}_i\right)}-1\Big|
=0.
\end{align}
Now we claim that
\begin{align}\label{eq:ClaimBi}
\lim_{n\rightarrow\infty}\sup_{\bar{c}:\|\bar{c}\|\le \omega_n}\Big|\pr(\bigcap_{1\le i\le k} \mathcal{B}_i|\Psi_n^{\rm Row}(\Cmatrix^k)=W)-1\Big|=0.
\end{align}
Indeed the event $\mathcal{B}_i, i\le k$ conditioned on $\Psi_n^{\rm Row}(\Cmatrix^k)=W$ is
\begin{align*}
L_{\ell+1,1}k^{1\over 2}w_{i,n}+L_{\ell+1,2}Z_{i,1}+\cdots L_{\ell+1,r+1}Z_{i,r}\le \sqrt{2\log n}+c_\ell, \qquad 1\le \ell\le r.
\end{align*}
Now recall from Lemma~\ref{lemma:CorrelatedColumns} that $L_{\ell+1,1}\le 1-1/k$. Then applying  (\ref{eq:scale of cn}) we conclude
\begin{align*}
L_{\ell+1,1}k^{1\over 2}w_{i,n}\le (1-1/k)\sqrt{2\log n}+o(1).
\end{align*}
Trivially, we have
\begin{align*}
\lim_{n\rightarrow\infty}
\pr\left(L_{\ell,2}Z_{i,1}+\cdots L_{\ell,r+1}Z_{i,r}\le {1\over 2k}\sqrt{2\log n},~\forall~ 1\le i\le k,~1\le \ell\le r\right)=1,
\end{align*}
simply because $\sqrt{\log n}$ is a growing function and the elements of $L$ are bounded by $1$. The claim then follows since $|c_\ell|\le \omega_n=o(\sqrt{2\log n})$. Similar to the reasoning of (\ref{eq:ClaimBi}), we also have
\begin{align}
\label{eq:ClaimBiUnconditional}
\lim_{n\rightarrow\infty}\sup_{\bar{c}:\|\bar{c}\|\le \omega_n}\Big|\pr(\bigcap_{1\le i\le k} \mathcal{B}_i)-1\Big|=0.
\end{align}
Then if we multiply the denominator of the first term in (\ref{eq:ClaimBi0}) by $\pr(\bigcap_{1\le i\le k} \mathcal{B}_i)$, we still have
\begin{align}
\label{eq:ClaimBi01}
\lim_{n\rightarrow\infty}\sup_{\bar{c}:||\bar{c}\|\le \omega_n}
\Big|{\pr\left(\bigcap_{k+1\le i\le n} \mathcal{B}_i |\max_{k+1\le i\le n}\text{\rm Ave}(\Cmatrix_{i\cdot}^{n\times k})\le c_n(W)\right)
\over \pr\left(\bigcap_{1\le i\le n} \mathcal{B}_i\right)}-1\Big|
=0.
\end{align}
Applying (\ref{eq:ClaimBi}) and (\ref{eq:ClaimBi01}) for (\ref{eq:CapBi})  we obtain (\ref{eq:BiConditioned}).

\end{proof}

\subsection{Bounding the number of steps of $\LAS$. Proof of Theorem~\ref{theorem:maintheorem}}
Next we obtain an upper bound on the number of steps taken by the $\LAS$ algorithm as well as a bound on the average value of the matrix $\Cmatrix_r^n$ obtained by the
$\LAS$ algorithm in step $r$, when $r$ is constant, and use these bounds to conclude the proof of Theorem~\ref{theorem:maintheorem}. For this purpose, we will rely on a repeated application of Theorem \ref{theorem:RD-CD_conditional}.

We now introduce some additional notations.
Fix  $r$ and consider the matrix $\Cmatrix_{2r}^n=\Cmatrix_{I_r^n,J_r^n}^n$ obtained in step $2r$ of $\LAS$, assuming $T_\LAS\ge 2r$. Recall $\tilde I_{r-1}^n$ is the set of $k$ rows with largest sum of entries in $\Cmatrix_{[n]\setminus I_{r-1}^n, J_r^n}$. Then the matrix $\Cmatrix_{2r}^n$ is obtained
by combining top rows of $\Cmatrix_{2r-1}^n=\Cmatrix_{I_{r-1}^n,J_r^n}^n$ and the top rows of $\Cmatrix_{\tilde I_{r-1}^n,J_r^n}^n$. We denote the part of $\Cmatrix_{2r}^n=\Cmatrix_{I_r^n,J_r^n}^n$ coming from $\Cmatrix_{I_{r-1}^n,J_r^n}^n$ by
$\Cmatrix_{2r,1}^n$ and the part coming from $\Cmatrix_{\tilde I_{r-1}^n,J_r^n}^n$ by $\Cmatrix_{2r,2}^n$.
The rows of $\Cmatrix_{I_{r-1}^n,J_r^n}^n$ leading
to $\Cmatrix_{2r,1}^n$ are denoted by $I_{r,1}^n\subset I_{r-1}^n$ with $|I_{r,1}^n|\triangleq K_1$ (a random variable), and
the rows of $\Cmatrix_{\tilde I_{r-1}^n,J_r^n}^n$ leading
to $\Cmatrix_{2r,2}^n$ are denoted by $I_{r,2}^n\subset \tilde I_{r-1}^n$ with $|I_{r,2}^n|\triangleq K_2=k-K_1$. Thus $I_{r,1}^n\cup I_{r,2}^n=I_r^n$ and
$\Cmatrix_{2r,\ell}^n=\Cmatrix_{I_{r,\ell}^n,J_r^n}^n, \ell=1,2$, as shown in Figure \ref{fig:ISPstep2r} where the symbol `$\triangle$' represents the entries in $\Cmatrix_{2r}^n$. Our first step is to show that starting from $r=2$, for every positive real $a$  the
average value of $\Cmatrix_r^n$ is at least $\sqrt{2\log n\over k}+a$ with probability bounded away from zero as $n$ increases. We will only show this
result for odd $r$ since by monotonicity we also have $\Ave(\Cmatrix_{r+1}^n)\ge \Ave(\Cmatrix_r^n)$.

\begin{figure}[!ht]
  \centering
  \includegraphics[width=0.6\textwidth]{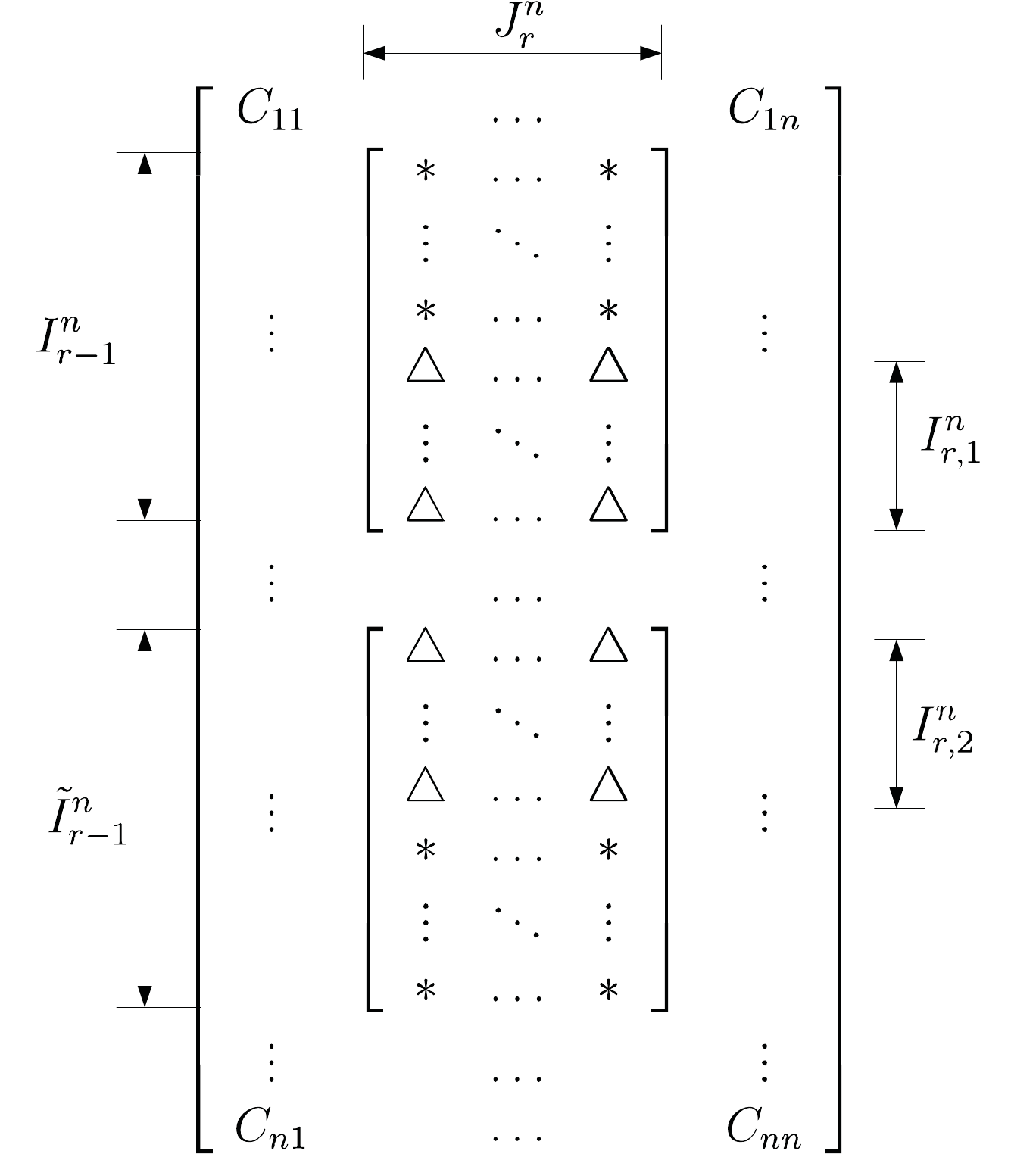}
  \caption{Step $2 r$ of $\LAS$ algorithm } \label{fig:ISPstep2r}
\end{figure}

\begin{prop}\label{prop:CmatrixPositive}
There exists a strictly positive function $\psi_1:\R_+\rightarrow\R_+$ which depends only on $k$, such that for all $r>0,a>0$
\begin{align*}
\liminf_n&\pr\left({\rm Ave}(\Cmatrix_{2r+1}^n)\ge \sqrt{2\log n\over k}+a \cup \{T_\LAS\le 2r\} |
T_\LAS\ge 2r-1\right)
\ge \psi_1(a).
\end{align*}
\end{prop}
Namely, assuming the algorithm proceeds for $2r-1$ steps, with probability at least approximately $\psi_1(a)$ either it stops in step $2r$ or proceeds to step $2r+1$, producing a matrix with average at least $\sqrt{2\log n/ k}+a$.

\begin{proof}
By Theorem~\ref{theorem:RD-CD_conditional}
the distribution of $\Psi_r^{\rm Row}(\Cmatrix_{\tilde I_{r-1}^n,J_r^n}^n)$ conditional on the event
$T_\LAS\ge 2r-1$
is given by $\Cmatrix_\infty^{\rm Row}$ in the limit as $n\rightarrow\infty$.
In particular, the row averages $\Ave(\Cmatrix_{i,J_r^n}^n), i\in \tilde I_{r-1}^n$ of this matrix
are concentrated around $\sqrt{2\log n\over k}$ w.h.p. as $n\rightarrow\infty$. Motivated by this we write
the row averages of
$\Cmatrix_{\tilde I_r^n,J_r^n}^n$ as $\sqrt{2\log n\over k}+C_1/(\sqrt{2k\log n}),\ldots,\sqrt{2\log n\over k}+C_k/(\sqrt{2k\log n})$ for the appropriate values $C_1,\ldots,C_k$.
Denote the event $\max_j |C_j|\le \omega_n$ by $\mathcal{L}_{2r}$. Then by Theorem~\ref{theorem:RD-CD_conditional} we have
\begin{align}
\lim_{n\rightarrow\infty} \pr\left(\mathcal{L}_{2r}^c|T_\LAS\ge 2r-1\right)=0. \label{eq:EventL}
\end{align}
If the event $T_\LAS\le 2r-1$ takes place then also $T_\LAS\le 2r$.
Now consider the event $T_\LAS\ge 2r$.
On this event the matrices $\Cmatrix_{2r,1}^n$ and $\Cmatrix_{2r,2}^n$  are well defined.
Recall the notations $I_{r,1}^n$ and $I_{r,2}^n$ for the row indices of $\Cmatrix_{2r,1}^n$ and $\Cmatrix_{2r,2}^n$ respectively,
and $0\le K_1\le k-1$ and $K_2=k-K_1$ -- their respective cardinalities. Suppose first that
\begin{align}
{\rm Sum}\left(\Cmatrix_{2r,1}^n\right)> K_1\sqrt{2k\log n}+2k^2a. \label{eq:SumCase1}
\end{align}
Then by the bound $\max_j |C_j|\le \omega_n$ where we recall $\omega_n = o(\sqrt{\log n})$ we have
\begin{align*}
{\rm Sum}\left(\Cmatrix_{2r}^n\right)&\ge (K_1+K_2)\sqrt{2k\log n}+2k^2a-K_2k\omega_n/\sqrt{2k\log n} \\
&\ge k^2\sqrt{2\log n\over k}+k^2a,
\end{align*}
for large enough $n$, implying $\Ave\left(\Cmatrix_{2r}^n\right)\ge \sqrt{2\log n\over k}+a$ and therefore either $\Ave\left(\Cmatrix_{2r+1}^n\right)\ge \sqrt{2\log n\over k}+a$
for large enough $n$ or $T_\LAS\le 2r$.

Now instead assume the event
\begin{align}
{\rm Sum}\left(\Cmatrix_{2r,1}^n\right)\le K_1\sqrt{2k\log n}+2k^2a,\label{eq:Cmatrix less a}
\end{align}
takes place (including the possibility $K_1=0$) which we denote by  $\mathcal{H}_1$.
Then there exists $j_0\in J_r^n$ such that
\begin{align*}
{\rm Sum}\left(\Cmatrix_{I_{r,1}^n,j_0}^n\right)&\le K_1\sqrt{2\log n\over k}+2ka.
\end{align*}
We pick any such column $j_0$, for example the one which is the smallest index-wise.
Consider the event
\begin{align*}
{\rm Sum}\left(\Cmatrix_{I_{r,2}^n,j_0}^n\right)&\le K_2\sqrt{2\log n\over k}-4k^2a.
\end{align*}
which we denote by $\mathcal{H}_2$.

We claim that the probability of the event $\mathcal{H}_2$ conditioned on the events $T_\LAS\ge 2r,\mathcal{L}_{2r}$ and
$\mathcal{H}_1$ is bounded away from zero as $n$ increases:
\begin{align*}
\liminf_n \pr\left(\mathcal{H}_2 | T_\LAS\ge 2r,\mathcal{L}_{2r},\mathcal{H}_1\right)>0.
\end{align*}
For this purpose fix any realization of the matrix $\Cmatrix_{2r-1}^n$
which we write as $\sqrt{2\log n\over k}+C$ for an appropriate $k\times k$ matrix $C$, the realizations $c_1,\ldots,c_k$ of $C_1,\ldots,C_k$, and the
realization $j_0\in J_r^n$,
which are all consistent with the events $T_\LAS\ge 2r,\mathcal{L}_{2r},\mathcal{H}_1$. In particular  the  row averages of
$\Cmatrix_{\tilde I_{r-1}^n,J_r^n}^n$ are $\sqrt{2\log n\over k}+c_1/(\sqrt{2k\log n}),\ldots,\sqrt{2\log n\over k}+c_k/(\sqrt{2k\log n})$ and $\max_j |c_j|\le \omega_n$.
Note that  $C$ and $c_1,\ldots,c_k$ uniquely determine the subsets $I_{r,1}^n$ and $I_{r,2}^n$, and their cardinalities
which we denote by $I_1,I_2$ and $k_1,k_2$ respectively.
Additionally, $c_1,\ldots,c_k$ uniquely determine $\Ave(\Cmatrix_{\tilde I_{r-1}^n,J_r^n}^n)$:
\begin{align*}
\Ave(\Cmatrix_{\tilde I_{r-1}^n,J_r^n}^n)=\sqrt{2\log n\over k}+{\sum c_j\over k\sqrt{2k\log n}},
\end{align*}
which we can also write as $\Ave(\Cmatrix_{\tilde I_{r-1}^n,J_r^n}^n)=\bar c/(\sqrt{2k\log n})+b_n/\sqrt{k}$
where $\bar c\triangleq \Psi_{n,1}^{\rm Row}\left(\Cmatrix_{\tilde I_{r-1}^n,J_r^n}^n\right)$.
Note that $\max_j |c_j|\le \omega_n=o(\sqrt{\log n})$ also implies $\bar c=o(\sqrt{\log n})$.
Next we show that
\begin{align}
\lim_{n\rightarrow\infty}\inf_{C,c_1,\ldots,c_k}\pr\left(\mathcal{H}_2|C,c_1\ldots,c_k\right)\ge \psi_1(a), \label{eq:Hconditional C c_1 c_k}
\end{align}
for some strictly positive function $\psi_1$ which depends on $k$ only, where
$\pr(\cdot | C,c_1,\ldots, c_k)$ indicates conditioning on the realizations $C,c_1,\ldots,c_k$ and $\inf_{C,c_1,\ldots,c_k}$ is taken
over all choices of $C,c_1,\ldots,c_k$ consistent with the events $T_\LAS\ge 2r,\mathcal{L}_{2r},\mathcal{H}_1$.
These realizations imply
\begin{align*}
\Psi_{n,2}^{\rm Row}\left(\Cmatrix_{\tilde I_{r-1}^n,J_r^n}^n\right)=\left(
                                                                   \begin{array}{c}
                                                                     c_1-\bar c \\
                                                                     \vdots \\
                                                                     c_k-\bar c \\
                                                                   \end{array}
                                                                 \right)
                                                                 \mathbf{1}' + \frac{\log (4\pi\log n)}{2}.
\end{align*}
where the last term is simply $\sqrt{2\log n}(\sqrt{2\log n} - b_n)$. Thus by representation (\ref{eq:AinTermsOfPsi}) and by $\bar c,c_j=o(\sqrt{\log n})$, we have
\begin{align*}
\Cmatrix_{\tilde I_{r-1}^n,J_r^n}^n &={\bar c \over \sqrt{2k\log n}}+{b_n \over \sqrt{k}}\mathbf{1}\mathbf{1}'+
(\sqrt{2k\log n})^{-1}\left(
                                                                   \begin{array}{c}
                                                                     c_1-\bar c \\
                                                                     \vdots \\
                                                                     c_k-\bar c \\
                                                                   \end{array}
                                                                 \right)
                                                                 \mathbf{1}' + {\log (4\pi\log n)\over 2\sqrt{2k\log n}}\\
&+\Psi_{n,3}^{\rm Row}\left(\Cmatrix_{\tilde I_{r-1}^n,J_r^n}^n\right)
+\Psi_{n,4}^{\rm Row}\left(\Cmatrix_{\tilde I_{r-1}^n,J_r^n}^n\right) \\
&=\sqrt{2\log n\over k}\mathbf{1}\mathbf{1}'
+\Psi_{n,3}^{\rm Row}\left(\Cmatrix_{\tilde I_{r-1}^n,J_r^n}^n\right)
+\Psi_{n,4}^{\rm Row}\left(\Cmatrix_{\tilde I_{r-1}^n,J_r^n}^n\right)
+O\left({\omega_n\over\sqrt{\log n}}\right),
\end{align*}
(recall that $\log\log n=O(\omega_n)$ and $\omega_n = o(\sqrt{\log n})$).
Then by Theorem~\ref{theorem:RD-CD_conditional} we have
\begin{align*}
\lim_{n\rightarrow\infty}\inf_{C,c_1,\ldots,c_k} \pr\left(\mathcal{H}_2 |C, c_1,\ldots,c_k\right)
\end{align*}
is the probability that the sum of the entries of ${\rm Col}(\Cmatrix^k)+\Anv(\Cmatrix^k)$ indexed by the subset $I_2$ and column $j_0$
is at most $-4k^2a$ which takes some value $\psi(a, |I_2|)>0$ and depends only on $a$, $k$ and the cardinality of $I_2$. Let $\psi_1(a) \triangleq \min_{1 \leq |I_2| \leq k} \psi(a, |I_2|)$, then the claime in (\ref{eq:Hconditional C c_1 c_k}) follows. We have established
\begin{align*}
\liminf_{n\rightarrow\infty} \pr\left(\mathcal{H}_2 | T_\LAS\ge 2r,\mathcal{L}_{2r},\mathcal{H}_1\right) \geq \psi_1(a).
\end{align*}

The event $\mathcal{H}_2$ implies that for some column $j_0$
\begin{align*}
{\rm Sum}\left(\Cmatrix_{I_{r}^n,j_0}^n\right)&\le K_1\sqrt{2\log n\over k}+2ka+K_2\sqrt{2\log n\over k}-4k^2a \\
&\le \sqrt{2k\log n}-3k^2a.
\end{align*}
By Theorem~\ref{theorem:RD-CD_conditional} conditional on all of the events
$T_\LAS\ge 2r,\mathcal{L}_{2r},\mathcal{H}_1,\mathcal{H}_2$,
every column average of
$\Cmatrix_{I_r^n,\tilde J_r^n}$ is concentrated around $\sqrt{2\log n\over k}$ w.h.p., implying that the column sum
is concentrated around $\sqrt{2k\log n}$ w.h.p..
Thus, w.h.p. the $j_0$-th column
will be replaced by one of the column in $\Cmatrix_{I_r^n,\tilde J_r^n}$ (and in particular $T_\LAS\ge 2r+1$)
and thus during the transition $\Cmatrix_{2r}^n\rightarrow \Cmatrix_{2r+1}^n$
the sum of the entries increases by $3k^2a-o(1)$, and thus the average {value} increases by at least $3a-o(1)$ w.h.p.
Recall from Theorem~\ref{theorem:RD-CD} that w.h.p.
$\Ave(\Cmatrix_{2r}^n)\ge \Ave(\Cmatrix_1^n)\ge \sqrt{2\log n\over k}-a$.
Then we obtain $\Ave(\Cmatrix_{2r+1}^n)\ge \sqrt{2\log n\over k} + 2a-o(1)\ge \sqrt{2\log n\over k}+a$ w.h.p.  We have obtained
\begin{align*}
\lim_{n\rightarrow\infty}\pr\left(\Ave(\Cmatrix_{2r+1}^n)\ge \sqrt{2\log n\over k}+a | T_\LAS\ge 2r,\mathcal{L}_{2r},\mathcal{H}_1,\mathcal{H}_2\right)=1.
\end{align*}
By earlier derivation we have
\begin{align*}
\liminf_{n\rightarrow\infty} \pr\left(\mathcal{H}_2 | T_\LAS\ge 2r,\mathcal{L}_{2r},\mathcal{H}_1\right) \geq \psi_1(a),
\end{align*}
thus implying
\begin{align*}
\liminf_n \pr\left(\Ave(\Cmatrix_{2r+1}^n)\ge \sqrt{2\log n\over k}+a | T_\LAS\ge 2r,\mathcal{L}_{2r},\mathcal{H}_1\right)\ge \psi_1(a).
\end{align*}
Next recall that $\mathcal{H}_1^c\cap \mathcal{L}_{2r}$ implies either $T_\LAS\le 2r$ or $\Ave(\Cmatrix_{2r+1}^n)\ge \sqrt{2\log n\over k}+a$ for large enough $n$, from which we obtain
\begin{align*}
\liminf_n \pr\left(\Ave(\Cmatrix_{2r+1}^n)\ge \sqrt{2\log n\over k}+a \cup \{T_\LAS\le 2r\} | T_\LAS\ge 2r-1,\mathcal{L}_{2r}\right)\ge \psi_1(a).
\end{align*}
Finally, recalling (\ref{eq:EventL}) we conclude
\begin{align*}
\liminf_n \pr\left(\Ave(\Cmatrix_{2r+1}^n)\ge \sqrt{2\log n\over k}+a \cup \{T_\LAS\le 2r\} | T_\LAS\ge 2r-1\right)\ge \psi_1(a).
\end{align*}
This concludes the proof of Proposition~\ref{prop:CmatrixPositive}.
\end{proof}

Now consider the event $T_\LAS\ge 2r$, and thus again $\Cmatrix_{2r,1}^n$ and $\Cmatrix_{2r,2}^n$ are well-defined. The definitions of $I_{r,1}^n,I_{r,2}^n$
and $K_1,K_2$ are as above. For any $a>0$
consider the event for every $j\in J_r^n$ the sum of entries of the column $j$ in $\Cmatrix_{2r,1}^n$ is at least $K_1\sqrt{2\log n\over k}-a$. Denote
this event by $\mathcal{F}_{2r}$.
Next we show that provided that $\Ave(\Cmatrix_{2r-1}^n)\ge \sqrt{2\log n \over k}+a$ with probability bounded away from zero as $n\rightarrow\infty$, for every fixed $r$,
either the event $\mathcal{F}_{2r+2t}$ takes place for some $t\le k$ or the algorithm stops earlier. To be more precise

\begin{prop}\label{prop:Below-a}
There exists a strictly positive function $\psi_2:\R_+\rightarrow\R_+$ which depends on $k$ only such that
for every $r>0$ and $a>0$
\begin{align*}
\liminf_{n\rightarrow\infty}
\pr&\left(\cup_{0\le t\le k}\left(\{T_\LAS\le 2r+2t-1\}\cup \mathcal{F}_{2r+2t}\right) | T_\LAS \ge 2r-1,
\Ave(\Cmatrix_{2r-1}^n)\ge \sqrt{2\log n\over k}+a \right)\\
&\ge \psi_2^{2(k+1)}(a).
\end{align*}
\end{prop}
The conditioning on the event $\Ave(\Cmatrix_{2r-1}^n)\ge \sqrt{2\log n\over k}+a $ will not be used explicitly below. The result just shows that even
with this conditioning, the claim still holds, so that this result can be used together with Proposition~\ref{prop:CmatrixPositive}.

\begin{proof}
On the event $T_\LAS\ge 2r-1$,
consider the event $\mathcal{G}_{2r}$ defined by
\begin{align}
\mathcal{G}_{2r}\triangleq \|\Cmatrix_{\tilde I_{r-1}^n,J_r^n}^n-\sqrt{2\log n\over k}\|_\infty \le {a\over 4k}. \label{eq:G2r+1}
\end{align}
Applying Theorem~\ref{theorem:RD-CD_conditional}, the distribution of $\Cmatrix_{\tilde I_{r-1}^n,J_r^n}^n$ conditioned on $T_\LAS\ge 2r-1$ and
$\Ave(\Cmatrix_{2r-1}^n)\ge \sqrt{2\log n\over k}+a$
is given asymptotically by $\Cmatrix_\infty^{\rm Row}$.
Recalling the representation (\ref{eq:AinTermsOfPsi}) we then have that for a certain strictly positive function $\psi_2$
\begin{align}
\liminf_n \pr\left(\mathcal{G}_{2r} | T_\LAS\ge 2r-1, \Ave(\Cmatrix_{2r-1}^n)\ge \sqrt{2\log n\over k}+a\right)\ge \psi_2(a). \label{eq:prG2r+1}
\end{align}
If $T_\LAS\le 2r-1$ then the event $\cup_{0\le t\le k}\left(\{T_\LAS\le 2r+2t-1\}\cup \mathcal{F}_{2r+2t}\right)$ holds as well. Otherwise
assume the event $T_\LAS\ge 2r$ takes place and then the matrices $\Cmatrix_{2r,1}^n$ and $\Cmatrix_{2r,2}^n$
which constitute $\Cmatrix_{2r}^n=\Cmatrix_{I_r^n,J_r^n}^n$
are well-defined.
If the event $\mathcal{F}_{2r}^c$ holds then there exists
$j_0\in J_r^n$, such that the sum of  entries of the column $\Cmatrix_{I_{r,1}^n,j_0}$ satisfies
\begin{align}
{\rm Sum}\left(\Cmatrix_{I_{r,1}^n,j_0}\right)<|I_{r,1}^n|\sqrt{2\log n\over k}-a. \label{eq:Column j0}
\end{align}
The event $\mathcal{G}_{2r}$ implies that the sum of entries of the column $\Cmatrix_{I_{r,2}^n,j_0}^n$ is at most
$|I_{r,2}^n|\sqrt{2\log n\over k}+a/4$, implying that the sum of entries of the column $\Cmatrix_{I_{r}^n,j_0}^n$ is at most
\begin{align}
|I_{r,1}^n|\sqrt{2\log n\over k}-a+|I_{r,2}^n|\sqrt{2\log n\over k}+a/4=\sqrt{2k\log n}-3a/4. \label{eq:Colum j0 Upper Bound}
\end{align}
Introduce now the event $\mathcal{G}_{2r+1}$ as
\begin{align}
\|\Cmatrix_{I_r^n,\tilde J_r^n}-\sqrt{2\log n\over k}\|_\infty \le {a\over 4k}. \label{eq:aover 4k}
\end{align}
Again applying Theorem~\ref{theorem:RD-CD_conditional}, we have that
\begin{align}
\liminf_n\pr\left(\mathcal{G}_{2r+1} | \mathcal{G}_{2r},T_\LAS\ge 2r,\mathcal{F}_{2r}^c,\Ave(\Cmatrix_{2r-1}^n)\ge \sqrt{2\log n\over k}+a
\right)\ge \psi_2(a), \label{eq:prG2r+2}
\end{align}
for the same function $\psi_2$. The event $\mathcal{G}_{2r+1}$ implies that the sum of entries of every column in matrix $\Cmatrix_{I_r^n,\tilde J_r^n}$
is in particular at least $\sqrt{2k\log n}-a/4$. Now recalling (\ref{eq:Colum j0 Upper Bound}) this implies that every column $\Cmatrix_{I_r^n,j_0}^n$
satisfying (\ref{eq:Column j0}) will be replaced by a new column from $\Cmatrix_{I_r^n,\tilde J_r^n}$ in the transition
$\Cmatrix_{2r}^n\rightarrow \Cmatrix_{2r+1}^n$ (and in particular this transition takes place and $T_\LAS\ge 2r+1$).
The event $\mathcal{G}_{2r+1}$ then implies that every column $\Cmatrix_{I_r^n,j_0}^n$
possibly contributing to the event $\mathcal{F}_{2r}^c$ is replaced by a new column in which every entry belongs to the interval
$[\sqrt{2\log n\over k}-a/(4k),\sqrt{2\log n\over k}+a/(4k)]$.

Now if $T_\LAS\le 2r+1$, then also $\cup_{0\le t\le k}\left(\{T_\LAS\le 2r+2t-1\}\cup \mathcal{F}_{2r+2t}\right)$. Otherwise, consider $T_\LAS\ge 2r+2$.
In this case we have a new matrix $\Cmatrix_{2r+2}^n$ consisting of $\Cmatrix_{2r+2,1}^n$ and $\Cmatrix_{2r+2,2}^n$.
Note that the event $\mathcal{G}_{2r+1}$ implies that for every subset $I\subset I_r^n$, and for every $j\in \tilde J_r^n$,
the sum of entries of the sub-column $\Cmatrix_{I,j}^n$ satisfies
\begin{align*}
{\rm Sum}\left(\Cmatrix_{I,j}^n\right)&\ge |I|\left(\sqrt{2\log n\over k}-a/(4k)\right) \\
& > |I|\sqrt{2\log n\over k}-a.
\end{align*}
In particular this holds for $I=I_{r+1,1}^n$ and therefore $j$ does not satisfy the property (\ref{eq:Column j0}) with $r+1$ replacing $r$.
Thus
the columns in $\Cmatrix_{I_{r+1,1}^n}^n$ satisfying (\ref{eq:Column j0}) with $r+1$ replacing $r$ can only be the columns which \emph{were not}
replaced in the transition $\Cmatrix_{2r}^n\rightarrow \Cmatrix_{2r+1}^n$. Therefore if the event $\mathcal{F}_{2r+2}^c$ takes place,
the columns contributing to this event are one of the original columns of $\Cmatrix_{2r}^n$.

To finish the proof we use a similar construction inductively and use the fact that the total number of original columns is at most $k$ and thus
after $2(k+1)$ iterations all of such columns will be replaced with columns for which (\ref{eq:Column j0}) cannot occur.
Thus assuming the events $\mathcal{G}_{2r},\ldots,\mathcal{G}_{2r+2t-1}$ are defined for some $t\ge 1$, on the event
$T_\LAS\ge 2r+2t-1$ we let
\begin{align*}
\mathcal{G}_{2r+2t}\triangleq \|\Cmatrix_{\tilde I_{r+t-1}^n,J_{r+t}^n}^n-\sqrt{2\log n\over k}\|_\infty \le {a\over 4k},
\end{align*}
and on the event $T_\LAS\ge 2r+2t$
\begin{align*}
\mathcal{G}_{2r+2t+1}\triangleq \|\Cmatrix_{I_{r+t}^n,\tilde J_{r+t}^n}^n-\sqrt{2\log n\over k}\|_\infty \le {a\over 4k}.
\end{align*}
Applying Theorem~\ref{theorem:RD-CD_conditional} we have for $t\ge 0$
\begin{align}
\liminf_n\pr\left(\mathcal{G}_{2r+2t} |\cdot\right)\ge \psi_2(a), \label{eq:psi2odd}
\end{align}
where $\cdot$ stands for conditioning on
$T_\LAS\ge 2r+2t-1, \Ave(\Cmatrix_{2r-1}^n)\ge \sqrt{2\log n\over k}+a$ as well
as
\begin{align*}
\left(\mathcal{G}_{2r}\cap\cdots\cap\mathcal{G}_{2r+2t-1}\right) \cap
\left(\mathcal{F}_{2r}^c\cap\cdots\cap\mathcal{F}_{2r+2t}^c\right)
\end{align*}
(here for the case $t=0$ the event above is assume to be the entire probability space and corresponds to the case considered above).
Similarly, for $t\ge 0$
\begin{align}
\liminf_n\pr\left(\mathcal{G}_{2r+2t+1} |\cdot\right)\ge \psi_2(a), \label{eq:psi2even}
\end{align}
where $\cdot$ stands for conditioning on
$T_\LAS\ge 2r+2t, \Ave(\Cmatrix_{2r-1}^n)\ge \sqrt{2\log n\over k}+a$ as well
as
\begin{align*}
\left(\mathcal{G}_{2r}\cap\cdots\cap\mathcal{G}_{2r+2t}\right) \cap
\left(\mathcal{F}_{2r}^c\cap\cdots\cap\mathcal{F}_{2r+2t}^c\right).
\end{align*}
By the observation above, since the total number of original columns of $\Cmatrix_{2r-1}^n$ is $k$, we have
\begin{align*}
\left(\mathcal{G}_{2r}\cap\cdots\cap\mathcal{G}_{2r+2(k+1)}\right) \cap
\left(\mathcal{F}_{2r}^c\cap\cdots\cap\mathcal{F}_{2r+2(k+1)}^c\right)=\emptyset.
\end{align*}

Iterating the relations (\ref{eq:psi2odd}),(\ref{eq:psi2even}), we conclude that conditional on the events
$T_\LAS\ge 2r-1, \Ave(\Cmatrix_{2r-1}^n)\ge \sqrt{2\log n\over k}+a$
with probability at least $\psi_2^{2(k+1)}(a)$ the event $\cup_{0\le t\le k}\left(\{T_\LAS\le 2r+2t-1\}\cup \mathcal{F}_{2r+2t}\right)$
takes place. This concludes the proof of the proposition.
\end{proof}

Our next step in proving  Theorem~\ref{theorem:maintheorem} is to show that if the events
${\rm Ave}(\Cmatrix_{2r-1}^n)\ge \sqrt{2\log n\over k}+a$ and $\mathcal{F}_{2r}$ take place (and in particular $T_{\LAS} \geq 2r$)
then with probability bounded away from zero
as $n\rightarrow\infty$ the algorithm actually stops in step $2r$: $T_\LAS\le 2r$.

On the event $T_\LAS\ge 2r-1$, the matrix $\Cmatrix_{\tilde I_r^n,J_r^n}^n$ is defined. As earlier, we write
the row averages of $\Cmatrix_{\tilde I_r^n,J_r^n}^n$ as
\begin{align*}
\sqrt{2\log n\over k}+C^n_1/(\sqrt{2k\log n}), \ldots, \sqrt{2\log n\over k}+C^n_k/(\sqrt{2k\log n}),
\end{align*}
for the appropriate values $C^n_1,\ldots,C^n_k$.
Denote the event $\max_j |C^n_j|\le \omega_n$ by $\mathcal{L}_{2r}$.
Then by Theorem~\ref{theorem:RD-CD_conditional}
\begin{align}
\lim_{n\rightarrow\infty} \pr\left(\mathcal{L}_{2r}^c|T_\LAS\ge 2r-1, \Ave(\Cmatrix_{2r-1}^n)\ge \sqrt{2\log n\over k}+a\right)=0. \label{eq:EventL2}
\end{align}
This observation will be used for our next result:
\begin{proposition}\label{proposition:LASStops}
There exists a strictly positive function $\psi_3:\R_+\rightarrow \R_+$ such that for every $r>0$ and $a>0$
\begin{align*}
\liminf_n \pr\left(T_\LAS \le 2r | T_\LAS\ge 2r, \mathcal{F}_{2r}, \mathcal{L}_{2r},\Ave(\Cmatrix_{2r-1}^n)\ge \sqrt{2\log n\over k}+a \right)\ge \psi_3(a).
\end{align*}
\end{proposition}

\begin{proof}
Consider any $k\times k$ matrix $C$, which is a realization of the matrix $\Cmatrix_{2r-1}^n-\sqrt{2\log n\over k}$ satisfying $\Ave(C) \ge a$,
namely consistent with the event $\Ave(\Cmatrix_{2r-1}^n)\ge \sqrt{2\log n\over k}+a$.
Note that the event $\Ave(\Cmatrix_{2r-1}^n)\ge \sqrt{2\log n\over k}+a$ implies that at least one of the row averages of
$\Cmatrix_{2r-1}^n$ is also at least $\sqrt{2\log n\over k}+a$. This event and the event $\mathcal{L}_{2r}$ then imply that for large enough $n$, at least one row of
$\Cmatrix_{2r-1}^n$ will survive till the next iteration $T_\LAS=2r$, provided that this iteration takes place, taking into account the realizations
of $C_1^n,\ldots,C_k^n$ corresponding to the row averages of $\Cmatrix_{\tilde I_{r-1}^n,J_r^n}^n$.

Now we assume that all of the events $T_\LAS\ge 2r, \mathcal{F}_{2r}, \mathcal{L}_{2r},\Ave(\Cmatrix_{2r-1}^n)\ge \sqrt{2\log n\over k}+a$
indeed take place. Consider
any constant $1\le k_1<k$ and the  subset $I\subset I_r^n$ with cardinality $k_1$ which corresponds to
the $k_1$ largest rows of $C$ with respect to row averages of $C$ (and therefore of $\Cmatrix_{2r-1}^n$ as well).
Let $A_1,\ldots,A_k$ be the column sums of the $k_1\times k$ submatrix of $C$
indexed by the rows $I$. Assume $A_1,\ldots,A_k\ge -a$. Consider the event that
 $I=I_{2r,1}^n$ corresponds precisely to the rows of $\Cmatrix_{2r-1}^n$ which survive in the next iteration. Then the
column sums of $\Cmatrix_{2r,1}^n$ are $k_1\sqrt{2\log n\over k}+A_j, 1\le j\le k$ consistently with  the event $\mathcal{F}_{2r}$.
Note that the lower bound  $\Ave(C)\ge a$ and the fact that the $k_1$ row selected are the largest $k_1\ge 1$ rows in $C^n$ implies
\begin{align}
\sum_{1\le j\le k}A^n_j \ge k_1 a \ge a. \label{eq:Sum aj biggger than a}
\end{align}
In order for the event above to take place it should be the case
that indeed precisely $k_2=k-k_1<k$ rows of $\Cmatrix_{\tilde I_{r-1}^n,J_r^n}^n$ will be
used in creating $\Cmatrix_{2r}^n$ with the corresponding
subset $I_{2r,2}^n, |I_{2r,2}^n|=k_2$. We denote this event by $\mathcal{K}_{k_2}$. Note that whether this event takes place is completely determined by the realization
$C$ corresponding to the matrix $\Cmatrix_{2r-1}^n$, in particular the realization of the row averages of this matrix, and the realizations
$C_1,\ldots,C_k$ of $C_1^n, \ldots, C_k^n$ corresponding to the row averages of $\Cmatrix_{\tilde I_{r-1}^n,J_r^n}^n$. Furthermore, the realizations
$C,C_1,\ldots,C_k$ determine the values $A_1,\ldots,A_k$.

We write the $k$ column sums of $\Cmatrix_{2r,2}^n$ as
$k_2\sqrt{2\log n\over k}+U_j^n,~ 1\le j\le k$. Then the column sums of $\Cmatrix_{2r}^n$ are $\sqrt{2k\log n}+U_j^n+A^n_j, 1\le j\le k$.
We claim that for a certain strictly positive function $\psi_3$ which depends on $k$ only these column sums are all at least
$\sqrt{2k\log n}+a/(2k)$:
\begin{align*}
\liminf_n \inf\pr\left(\sqrt{2k\log n}+U_j^n+A^n_j\ge \sqrt{2k\log n}+a/(2k),~j=1,\ldots,k ~|~
C^n, C_1^n,\ldots,C_k^n\right) \ge \psi_3(a),
\end{align*}
where $\inf$ is over all sequences $C, C_1,\ldots,C_k$ consistent with the events
$T_\LAS\ge 2r, \mathcal{F}_{2r}, \mathcal{L}_{2r},\Ave(\Cmatrix_{2r-1}^n)\ge \sqrt{2\log n\over k}+a$.
We first show how this claim implies the claim of the proposition.
The claim implies that conditional on the realizations of $C$, $C_1, \ldots, C_k$ these column sums are at least $\sqrt{2k\log n}+a/(2k)$ with probability $\psi_3(a)-o(1)$. By Theorem~\ref{theorem:RD-CD_conditional}
conditional on $\Cmatrix_{2r}^n$,
the column sums of
$\Cmatrix_{I_r^n,\tilde J_r^n}^n$ are concentrated around $\sqrt{2k\log n}$ w.h.p. Thus with high probability all columns of $\Cmatrix_{2r}^n$ dominate
the columns of $\Cmatrix_{I_r^n,\tilde J_r^n}^n$ by at least an additive factor $a/(2k)-o(1)$ and therefore algorithm stops at $T_\LAS=2r$. Integrating
over  $k_2=0,\ldots,k-1$ and realizations $C,C_1,\ldots,C_k$ consistent with the events
$T_\LAS\ge 2r, \mathcal{F}_{2r}, \mathcal{L}_{2r},\Ave(\Cmatrix_{2r-1}^n)\ge \sqrt{2\log n\over k}+a$
we obtain the result.

Thus it remains to establish the claim. We have
\begin{align*}
\pr&\left(\sqrt{2k\log n}+U_j^n+A^n_j\ge \sqrt{2k\log n}+a/(2k),~j=1,\ldots,k ~|~
C, C_1,\ldots,C_k\right) \\
&=\pr\left(U_j^n+A^n_j\ge a/(2k),~j=1,\ldots,k ~|~
C, C_1,\ldots,C_k\right) .
\end{align*}
Let $\hat A^n_j=\min(A^n_j, 2ka)$. Then
\begin{align*}
\pr&\left(U_j^n+A^n_j\ge a/(2k),~j=1,\ldots,k ~|~
C, C_1,\ldots,C_k\right) \\
&\geq\pr\left(U_j^n+\hat A^n_j\ge a/(2k),~j=1,\ldots,k ~|~
C, C_1,\ldots,C_k\right) .
\end{align*}
The event $\mathcal{L}_{2r}$ implies
that $\Psi_{n,1}^{\text{\rm Row}}(\Cmatrix_{\tilde I_{r-1}^n,J_r^n}^n)=o(\sqrt{\log n})$ and thus $\Psi_{n,1}^{\text{\rm Row}}(\Cmatrix_{\tilde I_{r-1}^n,J_r^n}^n)/\sqrt{2\log n}=o(1)$.
By a similar reason $\Psi_{n,2}^{\text{\rm Row}}(\Cmatrix_{\tilde I_{r-1}^n,J_r^n}^n)/\sqrt{2\log n}=o(1)$ thus implying from (\ref{eq:AinTermsOfPsi}) that
\begin{align*}
\Cmatrix_{\tilde I_{r-1}^n,J_r^n}^n=\sqrt{2\log n\over k}+\Psi_{n,3}^{\text{\rm Row}}(\Cmatrix_{\tilde I_{r-1}^n,J_r^n}^n)
+\Psi_{n,4}^{\text{\rm Row}}(\Cmatrix_{\tilde I_{r-1}^n,J_r^n}^n)+o(1)
\end{align*}
Then by Theorem~\ref{theorem:RD-CD_conditional} we have that
\begin{align*}
\lim_{n\rightarrow\infty}\sup_{C, C_1,\ldots,C_k} &\Big|\pr\left(U_j^n+\hat A^n_j\ge a/(2k), ~j=1,\ldots,k ~|~ C, C_1,\ldots,C_k\right)  \\
&-\pr\left(U_j+\hat A^n_j\ge a/(2k), ~ j = 1, \ldots, k| \hat A^n_1,\ldots,\hat A^n_k\right)\Big|=0,
\end{align*}
where $U_j$ is the $j$-th column sum of the matrix of the $k_2\times k$ submatrix of ${\rm Col}(\Cmatrix^k)+\Anv(\Cmatrix^k)$ indexed
by $I_{r,2}^n$ and $\sup_{C, C_1,\ldots,C_k}$ is over the realizations $C, C_1,\ldots,C_k$ consistent with
$T_\LAS\ge 2r, \mathcal{F}_{2r}, \mathcal{L}_{2r},\Ave(\Cmatrix_{2r-1}^n)\ge \sqrt{2\log n\over k}+a$. Thus it suffice to show that
\begin{align*}
\inf_{\hat A^n_1,\ldots,\hat A^n_k}\pr\left(U_j+\hat A^n_j\ge a/(2k), ~ j = 1, \ldots, k| \hat A^n_1,\ldots,\hat A^n_k\right)\ge \psi_3(a),
\end{align*}
for some strictly positive function $\psi_3$ which depends on $k$ only,
where the infimum is over $\hat A^n_1,\ldots,\hat A^n_k$ satisfying $-a\le \hat{A}^n_j\le 2ka$ and (\ref{eq:Sum aj biggger than a}).
The joint distribution of $U_j, 1\le j\le k$ is the one of $\left(\sqrt{k_2}(Z_j-\bar Z), 1\le j\le k\right)$ where $Z_1,\ldots,Z_k$ are i.i.d. standard normal
and $\bar Z=k^{-1}\sum_{1\le j\le k} Z_j$. Thus our goal is to show that
\begin{align*}
\inf_{\hat A^n_1,\ldots,\hat A^n_k}\pr\left(\sqrt{k_2}(Z_j-\bar Z)+\hat A^n_j\ge a/(2k), 1\leq j \leq k | \hat A^n_1,\ldots,\hat A^n_k \right)\ge \psi_3(a),
\end{align*}
for some $\psi_3$.
The distribution of the normal $\left(\sqrt{k_2}(Z_j-\bar Z),  j = 1, \ldots, k\right)$ vector has a full support on the set $\{x=(x_1,\ldots,x_k): \sum_j x_j=0\}$.

Consider the set of such vectors $x\in\R^k$ satisfying $\sum_j x_j=0$ and $x_j+\hat A^n_j\ge a/(2k)$. Denote this set by $X(\hat A^n_1,\ldots,\hat A^n_k)$.
By (\ref{eq:Sum aj biggger than a}) we have $\sum_j (a/(2k)-A^n_j)\le -a/2$. We claim that in fact
\begin{align}
\sum_j (a/(2k)-\hat A^n_j) \le  -a/2<0, \label{eq:Sum hat aj biggger than a}
\end{align}
and thus the set $X(\hat A^n_1,\ldots,\hat A^n_k)$ is non-empty. Indeed, if $A^n_j\le 2ka$, for all $j$ then $\hat A^n_j=A^n_j$ and assertion holds from (\ref{eq:Sum aj biggger than a}).
Otherwise, if  $A_{j_0}^n>2ka$ for some $j_0$, then since $A^n_j\ge -a$ and therefore $\hat A^n_j\ge -a$, we have
\begin{align*}
\sum_j (a/(2k)-\hat A^n_j) \le  a/2-2ka+(k-1)a< -ka <-a/2<0.
\end{align*}
In fact since $a>0$, the set $X(\hat A^n_1,\ldots,\hat A^n_k)$ has a non-empty interior and thus a positive measure with respect to the induced Lebesgue measure of the
subset $\{x=(x_1,\ldots,x_k): \sum_j x_j=0\}\subset \R^k$.  As a result the probability
\begin{align*}
\pr\left((\sqrt{k_2}(Z_j-\bar Z), 1\le j\le k) \in X(\hat A^n_1,\ldots,\hat A^n_k) | \hat A^n_1,\ldots,\hat A^n_k \right)
\end{align*}
is strictly positive. This probability is a continuous function of $\hat A^n_1, \ldots,\hat A^n_k$ which belong to the bounded interval $[-a,2ka]$. By compactness
argument we then obtain
\begin{align*}
\inf\pr\left( (\sqrt{k_2}(Z_j-\bar Z), 1\le j\le k) \in X(\hat A^n_1,\ldots,\hat A^n_k) | A^n_1,\ldots,A^n_k \right) >0,
\end{align*}
where the infimum is over $-a\le \hat A^n_1,\ldots,\hat A^n_k\le 2ka$ satisfying (\ref{eq:Sum hat aj biggger than a}).
Denoting the infimum  by $\psi_3(a)$ we obtain the result.
\end{proof}

We now synthesize Propositions~\ref{prop:CmatrixPositive},\ref{prop:Below-a} and \ref{proposition:LASStops} to obtain the following corollary.
\begin{coro}\label{corollary:synthesis}
There exists a strictly positive function $\psi_4$ which depends on $k$ only such that for every $r>k+2$ and $a>0$
\begin{align*}
\liminf_n \pr\left(T_\LAS\le 2r| T_\LAS \ge 2r-2k-3\right)\ge \psi_4(a).
\end{align*}
\end{coro}

\begin{proof}
By Proposition~\ref{prop:CmatrixPositive}, we have
\begin{align*}
\liminf_n \pr\left(\Ave(\Cmatrix_{2r-2k-1})\ge \sqrt{2\log n\over k}+a \cup \{T_\LAS\le 2r-2k-2 \}| T_\LAS \ge 2r-2k-3\right)\ge \psi_1(a).
\end{align*}
Combining with Proposition~\ref{prop:Below-a}, we obtain that there exists $t, 0\le t\le k$ such that
\begin{align*}
\liminf_{n\rightarrow\infty}
\pr&\left(\{T_\LAS\leq 2r-2t-1\}\cup \left(\mathcal{F}_{2r-2t} \cap \Ave(\Cmatrix_{2r-2t-1})\ge \sqrt{2\log n\over k}+a\right) ~|~ T_\LAS \ge 2r-2k-3\right)\\
&\ge (k+1)^{-1}\psi_1(a)\psi_2^{2(k+1)}(a).
\end{align*}
By observation (\ref{eq:EventL2}) we also obtain
\begin{align*}
\liminf_{n\rightarrow\infty}
\pr&\left(\{T_\LAS\leq 2r-2t-1\}\cup\left(\mathcal{F}_{2r-2t}\cap \Ave(\Cmatrix_{2r-2t-1})\ge \sqrt{2\log n\over k}+a\cap \mathcal{L}_{2r-2t}\right)| T_\LAS \ge 2r-2k-3\right)\\
&\ge (k+1)^{-1}\psi_1(a)\psi_2^{2(k+1)}(a).
\end{align*}
Finally, applying Lemma~\ref{proposition:LASStops} we obtain
\begin{align*}
\liminf_{n\rightarrow\infty}
\pr&\left(\{T_\LAS\le 2r-2t\} |  T_\LAS \ge 2r-2k-3\right)
\ge (k+1)^{-1}\psi_1(a)\psi_2^{2(k+1)}(a)\psi_3(a),
\end{align*}
implying by monotonicity the same result for $T_\LAS\le 2r$. Letting $\psi_4(a)\triangleq (k+1)^{-1}\psi_1(a)\psi_2^{2(k+1)}(a)\psi_3(a)$, we obtain the result.
\end{proof}

We are now ready to complete the proof of Theorem~\ref{theorem:maintheorem}.

\begin{proof}[Proof of Theorem~\ref{theorem:maintheorem}]
Given $\epsilon>0$ we fix arbitrary $a>0$ and find $r=r(\epsilon,a)$ large enough so that
$(1-\psi_4(a))^r<\epsilon$. Applying Corollary~\ref{corollary:synthesis} we obtain for $N=r(2k+4)$
\begin{align*}
\pr\left(T_\LAS \ge N \right)&=\prod_{1\le t\le r}\pr\left(T_\LAS \ge t(2k+4)|T_\LAS \ge (t-1)(2k+4)\right) \\
&\le (1-\psi_4(a))^r \\
&\le \epsilon,
\end{align*}
which gives the first part of Theorem~\ref{theorem:maintheorem}. We now show (\ref{eq:mainresult}). Fix $\epsilon>0$.
We have
\begin{align*}
\mathbb{P}\left(\big| \Ave(\Cmatrix_{T_{\LAS}}^n) - \sqrt{2\log n\over k} \big|> \omega_n \right)
&\le \mathbb{P}\left(\big| \Ave(\Cmatrix_{T_{\LAS}}^n) - \sqrt{2\log n\over k} \big|> \omega_n , T_{\LAS}^n\le N_\epsilon\right)
+\pr\left(T_{\LAS}^n> N_\epsilon\right) \\
&\le \mathbb{P}\left(\big| \Ave(\Cmatrix_{T_{\LAS}}^n) - \sqrt{2\log n\over k} \big|> \omega_n , T_{\LAS}^n\le N_\epsilon\right)+\epsilon \\
&=\sum_{1\le r\le N_\epsilon}
\mathbb{P}\left(\big| \Ave(\Cmatrix_r^n) - \sqrt{2\log n\over k} \big|> \omega_n , T_{\LAS}^n=r\right)+\epsilon \\
&\le \sum_{1\le r\le N_\epsilon}
\mathbb{P}\left(\big| \Ave(\Cmatrix_r^n) - \sqrt{2\log n\over k} \big|> \omega_n , T_{\LAS}^n\ge r\right)+\epsilon.
\end{align*}
By part (b) of Theorem~\ref{theorem:RD-CD_conditional}, we have for every $r$
\begin{align*}
\lim_{n\rightarrow\infty}
\mathbb{P}\left(\big| \Ave(\Cmatrix_r^n) - \sqrt{2\log n\over k} \big|> \omega_n , T_{\LAS}^n\ge r\right)=0.
\end{align*}
We conclude that for every $\epsilon$
\begin{align*}
\lim_{n\rightarrow\infty}\mathbb{P}\left(\big| \Ave(\Cmatrix_{T_{\LAS}}^n) - \sqrt{2\log n\over k} \big|> \omega_n \right)\le \epsilon.
\end{align*}
Since the left hand-side does not depend on $\epsilon$, we obtain (\ref{eq:mainresult}). This concludes the proof of Theorem~\ref{theorem:maintheorem}.
\end{proof}

\section{Conclusions and Open Questions}\label{section:Conclusions}
We close the paper with several open questions for further research. In light of the new algorithm $\IGP$ which improves upon the $\LAS$
algorithm by factor $4/3$, a natural direction is to obtain a better performing polynomial time algorithm. It would be especially interesting if such an algorithm
can improve upon the $5\sqrt{2}/3\sqrt{3}$ threshold since it would then indicate that the OGP is not an obstacle for polynomial
time algorithms. Improving the $5\sqrt{2}/3\sqrt{3}$ threshold perhaps by considering multi-overlaps of matrices with fixed asymptotic average value
is another important challenge. Based on such improvements obtainable for independent sets in sparse random random graphs~\cite{rahman2014local}
and for random satisfiability (random NAE-K-SAT) problem~\cite{gamarnik2014performance}, it is very plausible that such an improvement is achievable.

Studying the maximum submatrix problem for non-Gaussian distribution is another interesting directions, especially for distributions with tail behavior
different from the one of the normal distribution, namely for not sub-Gaussian distributions. Heavy tail distributions are of particular interest for this problem.

Finally, a very interesting version of the maximum submatrix problem is the sparse Principal Component Analysis (PCA) problem for sample covariance data.
Suppose, $X_i, 1\le i\le n$ are $p$-dimensional uncorrelated random variables (say Gaussian), and let $\Sigma$ be the corresponding sample covariance matrix.
When the dimension $p$ is comparable with $n$ the distribution of $\Sigma$ exhibits a non-trivial behavior. For example the limiting distribution of the spectrum
is described by the Marcenko-Pastur law as opposed to the ''true'' underlying covariance matrix which is just the identity. The sparse PCA problem is the
maximization problem $\max \beta^T\Sigma\beta$ where the maximization is over $p$-dimensional vectors $\beta$ with $\|\beta\|_2=1$ and $\|\beta\|_0=k$, where
$\|a\|_0$ is the number of non-zero components of the vector $a$ (sparsity). What is the limiting distribution of the objective value and what is the algorithmic
complexity structure of this problem? What is the solutions space geometry of this problem and in particular, does it exhibit the OGP? The sparse PCA problem
has received an attention recently in the hypothesis testing version~\cite{berthet2013complexity},\cite{berthet2013optimal}, where it was shown for certain
parameter regime, detecting the sparse PCA signal is hard provided the so-called Hidden Clique problem in the theory of random graphs is hard~\cite{alon1998finding}.
Here we propose to study the problem from the estimation point of view - computing the distribution of the $k$-dominating principal components and studying
the algorithmic hardness of this problem.

Finally, a bigger challenge is to either establish that the problems exhibiting the OGP are indeed algorithmically hard and do not admit a polynomial time algorithms,
or constructing an example where this is not the case. In light of the repeated failure to improve upon the important special case of this problem - largest clique
in the \ER graph $\G(n,p)$, this challenge might be out of reach for the existing methods of analysis.

\bibliographystyle{amsalpha}
%\bibliography{bibliography}
%\bibliography{Large_Submatrix_ref}
\bibliography{bibliography}

\end{document}